\documentclass[11pt, a4paper, twoside]{article}
\pdfoutput=1
\usepackage{lspaper}
\newgeometry{margin=.83in}
\usepackage{tikzsymbols}
\usepackage{booktabs}
\usepackage{multirow, multicol}
\usepackage{placeins}
\usepackage[normalem]{ulem}
\usepackage{amssymb,amsmath}
\usepackage{pdflscape,afterpage}
\usetikzlibrary{matrix,arrows,decorations.pathmorphing,decorations.pathreplacing}
\usetikzlibrary{positioning,intersections,shapes.geometric,calc}
\usetikzlibrary{decorations.markings}
\RequirePackage{pxfonts}
\usepackage[toc,page]{appendix}
\usepackage{scalefnt,caption}
\usepackage[all, knot]{xy}

\renewcommand\sss[1][n]{\mathfrak{S}_{#1}}
\renewcommand\spe[1]{\rspe{#1}}

\renewcommand\D[1]{\rD{#1}}
\newcommand{\Ext}{\operatorname{Ext}}
\newcommand{\Rad}{\operatorname{Rad}}

\newcommand{\lmod}{{\operatorname{-mod}}}

\renewcommand\phi\varphi

\newcommand{\vn}{\varnothing}
\newcommand{\rt}[1]{\rotatebox{90}{$#1$}}

\renewcommand{\theenumi}{\roman{enumi}}

\definecolor{answercolor}{RGB}{0, 112, 48}
\includecomment{answer}    

\begin{document}

\title{Wild blocks of type $A$ Hecke algebras are strictly wild}

\author{Liron Speyer
\\\normalsize Okinawa Institute of Science and Technology\\\normalsize Onna-son, Okinawa, Japan 904-0495 \\\texttt{\normalsize liron.speyer@oist.jp}
}

\renewcommand\auth{Liron Speyer}

\runninghead{Wild blocks of type $A$ Hecke algebras are strictly wild}


\toptitle

\begin{abstract}
We prove that all wild blocks of type $A$ Hecke algebras with quantum characteristic $e \geq 3$ -- i.e.~blocks of weight at least $2$ -- are \emph{strictly} wild, with the possible exception of the weight $2$ Rouquier block for $e = 3$.
As a corollary, we show that for $e \geq 3$, all wild blocks of the $q$-Schur algebras are strictly wild, without exception.
\end{abstract}

\section{Introduction}

Let $\bbf$ be an algebraically closed field of characteristic $p\geq 0$, and let $\hhh$ denote the type $A$ Iwahori--Hecke algebra, a deformation of the group algebra $\bbf\sss$ of the symmetric group $\sss$.
Erdmann and Nakano determined the representation type of blocks of $\hhh$ as follows.

\begin{thmc}{enreptype}{Main Result}\label{thm:enrep}
Let $B$ be a block of $\hhh$.
Then
\begin{itemize}
\item
$B$ is simple if and only if it has weight $0$.

\item $B$ has finite representation type if and only if it has weight $\leq 1$.

\item $B$ has tame representation if and only if it has quantum characteristic $e=2$ and weight $2$.

\item $B$ has wild representation type if and only if either (i) $B$ has weight $2$ and the quantum characteristic is $e\geq3$; or (ii) $B$ has weight $w \geq 3$ (for any quantum characteristic).
\end{itemize}
\end{thmc}

Recently, Ariki, Lyle and the author determined the Schurian-finiteness of blocks of $\hhh$ under the assumption that $e\geq 3$~\cite{als23}.
As Schurian modules are also known as bricks, Schurian-finiteness coincides with the notion of brick-finiteness, which is widely treated in the literature.
Our main result was that a block is brick-finite if and only if it has finite representation type.

Recall that an algebra $A$ has wild representation type if there exists a finitely-generated $A{-}\bbf \langle X, Y \rangle$-bimodule $M$ which is free as a right $\bbf \langle X, Y \rangle$-module, such that the functor $M \otimes_{\bbf \langle X, Y \rangle} - : \bbf \langle X, Y \rangle\lmod \rightarrow A\lmod$ preserves indecomposability and isomorphism classes.
We say that $A$ is \emph{strictly wild} if this functor is also full.

The purpose of this paper is to determine which blocks of $\hhh$ that are of wild representation type are in fact \emph{strictly} wild.
Being strictly wild is a nice property -- the fullness of the functor in the definition means that any homomorphism between modules in the image of the functor arises as the image of a homomorphism.
In fact, it is known that if $A$ is a strictly wild algebra, then for any algebra $B$, we have that $B\text{-}\mathrm{mod}$ is a full subcategory of $A\text{-}\mathrm{mod}$.
It also immediately implies brick-wildness, a strengthening of the notion of brick-infinite (or Schurian-infinite).
See \cref{sec:background} for the relevant definitions.
Our main result is the following theorem.

\begin{thm}\label{thm:MAIN}
Suppose $e\geq 3$, and that $B$ is a block of $\hhh$ of weight at least 2.
If $e=3$, suppose further that $B$ is not (Scopes equivalent to) the weight $2$ Rouquier block.
Then $B$ is strictly wild, and therefore brick-wild, in any characteristic.
\end{thm}

Our method of proving the theorem build on those in \cite{als23}, reducing the problem to determining certain submatrices of graded decomposition matrices -- see \cref{prop:matrixtrick}.
In the case missing in the theorem -- the weight $2$ Rouquier block for $e=3$ -- our method is insufficient for determining that the block is strictly wild.
We invite the expert reader to determine if this block is also strictly wild.

For the $q$-Schur algebra $S_q(n,n)$ as introduced by Dipper and James~\cite{dj89}, we may prove an analogous but stronger result, as we do not need to omit a case as above.

\begin{cor}\label{cor:schuralg}
Suppose $e\geq 3$, and that $B$ is a block of the $q$-Schur algebra $S_q(n,n)$, of weight at least $2$.
Then $B$ is strictly wild, and therefore brick-wild, in any characteristic.
\end{cor}

The paper is organised as follows.
In \cref{sec:background}, we recall the necessary definitions and results that we will use later.
In \cref{sec:reduction}, we give a reduction result, which distills the proof of \cref{thm:MAIN} down to finding certain submatrices of graded decomposition matrices.
In \cref{sec:wt2}, we prove \cref{thm:MAIN} for weight 2 blocks, while weight 3 blocks are handled in \cref{sec:wt3}.
Finally, the theorem is proved for blocks of weight at least 4 in \cref{sec:highwt}, with much of the work involving reducing down to blocks of weight 2 or 3 by row-removal.
We end with a proof of \cref{cor:schuralg}.

\begin{ack}
The author is partially supported by JSPS Kakenhi grant number 23K03043.
We thank Kaveh Mousavand for originally asking the question which led to this project, and for pointing out that the problem may be attacked via \cref{cor:subquiver}.
We also thank Karin Erdmann for her interest in the project, and for helpful conversations.
We thank Chris Bowman for a question about $q$-Schur algebras, prompting us to include \cref{cor:schuralg}, and Aaron Chan for communicating to us a presentation for the `missing' Rouquier block in the $p\neq2$ case.
Unfortunately, we were unable to determine from the latter whether the block is strictly wild or not and do not include the presentation here.
Finally, we thank the referee for their helpful comments and corrections.
\end{ack}

\section{Background}\label{sec:background}

Throughout, we let $\bbf$ denote an algebraically closed field of characteristic $p\geq 0$.
All our modules are left modules, considered up to isomorphism in the following definitions.

\begin{defn}
For a finite-dimensional algebra $A$ over $\bbf$, we say that $A$ is \emph{wild} if there exists a finitely-generated $A{-}\bbf \langle X, Y \rangle$-bimodule $M$ which is free as a right $\bbf \langle X, Y \rangle$-module, such that the functor $M \otimes_{\bbf \langle X, Y \rangle} - : \bbf \langle X, Y \rangle\lmod \rightarrow A\lmod$ preserves indecomposability and isomorphism classes.
We say that $A$ is \emph{strictly wild} if this functor is also full.
\end{defn}

\begin{rem}
We may equivalently define $A$ to be wild if there exists an exact faithful functor\linebreak $\bbf \langle X, Y \rangle\lmod \rightarrow A\lmod$ which preserves indecomposability and isomorphism classes, and strictly wild if there exists such a functor which is also full.
We note that such functors are often called \emph{representation embeddings} and \emph{full representation embeddings} in the literature.
\end{rem}

\begin{defn}
A module $M$ over an algebra $A$ is called a \emph{brick} (or \emph{Schurian}) if $\End_A(M) \cong \bbf$.

An algebra $A$ is called \emph{brick-finite} (or \emph{Schurian-finite}) if it has finitely many isomorphism classes of bricks, and \emph{brick-infinite} (or \emph{Schurian-infinite}) otherwise.

An algebra $A$ is called \emph{brick-tame} if in each dimension, all but finitely many bricks are in finitely many one-parameter families.
We call $A$ \emph{brick-wild} otherwise.
\end{defn}

More details about strictly wild algebras can be found in \cite[Section~XVIII.4]{simsonskow3}, while details about brick-tameness may be found in \cite{BDbricktame} and \cite[Section~2.3]{CCSchurtame}.
We remark that we will not directly use the definition of brick-tame or brick-wild above, but include it as a refinement of the brick-infinite property shown to hold for blocks of Hecke algebras in \cite{als23}.
In particular, as shown in \cite[Remark 1]{CCSchurtame}, any strictly wild algebra is also brick-wild.
In this paper, we will show that most blocks of Hecke algebras of weight at least 2 are strictly wild, and therefore brick-wild.
This may be seen as a simultaneous strengthening of Erdmann--Nakano's result \cref{thm:enrep}, that blocks of weight at least 2 have wild representation type, and the main result of \cite{als23}, that blocks of weight at least 2 are brick-infinite.

Next, we introduce the type $A$ Hecke algebras.
Let $q\in \bbf^\times$.
The Hecke algebra of the symmetric group, denoted by $\hhh$, is the unital associative $\bbf$-algebra generated by $T_1, T_2, \dots, T_{n-1}$ subject to the relations
\begin{alignat*}3
(T_i - q)(T_i + 1) &= 0 \qquad &&\text{ for } i = 1,\dots, n-1;\\
T_i T_j &= T_j T_i \qquad &&\text{ for } 1\leq i,j \leq n-1 \text{ with } |i-j|>1;\\
T_i T_{i+1} T_i &= T_{i+1} T_i T_{i+1} \qquad &&\text{ for } i = 1,\dots, n-2.
\end{alignat*}

The \emph{quantum characteristic} of $\hhh$ is the smallest positive integer $e$ such that $1 + q + q^2 + \dots + q^{e-1} = 0$, if such an $e$ exists, and we set $e = \infty$ otherwise.

Mathas's book~\cite{mathas} serves as a good introduction to the representation theory of these algebras, and we omit many details here for brevity.
We will assume that the reader is familiar with partitions, beta numbers, and abacus displays.
The reader may refer to \cite[Sections~2.1 and 2.3]{als23} for these standard definitions.
We recall from \cite[Section~2.3]{als23} that the integers $p_0 < p_1 < \dots < p_{e-1}$ for a core partition $\rho$ are defined so that each is the position of the lowest bead on one of the runners of the abacus display for $\rho$.

The algebra $\hhh$ has Specht modules indexed by partitions of $n$ and simple modules indexed by $e$-regular partitions of $n$.
Two simple (or Specht) modules are in the same block if and only if their indexing partitions have the same $e$-cores and weights.
Thus we will denote by $B(\rho,w)$ the block of $\hhh[|\rho|+ew]$ with core $\rho$ and weight $w$.
More details may be found in \cite[Section~2.2]{als23}.

It is also known that the algebras $\hhh$ may be graded, via the Brundan--Kleshchev isomorphism between blocks of $\hhh$ and cyclotomic KLR algebras.

Let $\la, \mu$ be partitions of $n$, with $\mu$ $e$-regular.
The corresponding \emph{decomposition number} is the multiplicity $d_{\la\mu}^{e,p}(1) = [\spe\la : \D\mu]$ of $\D\mu$ in $\spe\la$.
For a graded $\hhh$-module $D$, we let $D\langle d \rangle$ denote the graded shift (by $d$) of the module $D$ -- in other words $D\langle d \rangle_r = D_{r-d}$.
Then the corresponding \emph{graded} decomposition number is the Laurent polynomial
\[
d_{\la\mu}^{e,p}(v) = [\spe\la : \D\mu]_v = \sum_{d\in\bbz} [\spe\la : \D\mu \langle d \rangle] v^d \in \bbn[v,v^{-1}].
\]
It is known that $d_{\la\la}^{e,p}(v) = 1$ and $d_{\la\mu}^{e,p}(v) \neq 0$ only if $\la \domby \mu$.

By~\cite[Theorem~5.17]{bk09}, the characteristic $p$ graded decomposition matrix is obtained from the characteristic $0$ one by post-multiplication by the lower-unitriangular adjustment matrix, whose entries $a_{\la\mu}(v)$ -- indexed by $\la$ and $\mu$ both $e$-regular partitions of $n$ -- are Laurent polynomials with nonnegative integral coefficients, symmetric in $v, v^{-1}$.
In other words,
\[
d_{\la\mu}^{e,p}(v) = d_{\la\mu}^{e,0}(v) + \sum_{\nu \domsby \mu} d_{\la\nu}^{e,0}(v) a_{\nu\mu}(v).
\]

In computing graded decomposition numbers, we will use several reduction techniques.
The first is (graded) Scopes equivalence.
The reader may find a detailed account of this in \cite[Sections~2.7~and~3]{als23}.
Here we only briefly recount the necessary notions.

Let $B=B(\rho,w)$ be a block of $\hhh$, and look at the abacus display for $\rho$.
Suppose that runner $i$ (counting from the left) has $k$ more beads than runner $i-1$ for some $k\geq w$.
Then if we let $A = B(\Phi(\rho),w)$ denote the block of $\hhh[n-k]$ whose core is obtained from $\rho$ by swapping the runners in the abacus display.
Then defining the map $\Phi$ analogously for all partitions in $B$, we get that $\Phi$ is a bijection from partitions in $B$ to partitions in $A$, sending $e$-regular partitions to $e$-regular partitions.
Then the blocks $B$ and $A$ are (graded) Morita equivalent, and if $\la$ and $\mu$ are partitions of $n$, with $\mu$ $e$-regular, then for any $p\geq 0$, $d_{\la\mu}^{e,p}(v) = d_{\Phi(\la)\Phi(\mu)}^{e,p}(v)$.

Our next result is a runner removal theorem of James and Mathas, though we use a useful reformulation of the statement due to Fayers~\cite[Theorem~2.15]{fay07wt4}.

\begin{thmc}{jm02}{Theorem 3.2}\label{thm:runnerrem}
Suppose $e\geq 3$ and that $\la$ and $\mu$ are partitions of $n$, $\mu$ is $e$-regular, and that we take abacus displays for $\la$ and $\mu$.
Suppose that for some $i$, the last bead on runner $i$ occurs before every unoccupied space on both abacus displays, and define two abacus displays with $e-1$ runners by deleting runner $i$ from the abacus displays of $\la$ and $\mu$.
Let $\la^-$ and $\mu^-$ be the partitions defined by these displays.
If $\mu^-$ is $(e-1)$-regular, then
\[
d^{e,0}_{\la\mu}(v) = d^{e-1,0}_{\la^- \mu^-}(v).
\]
\end{thmc}

The following row-removal theorem will be used in \cref{sec:highwt} for determining that blocks of weight at least 4 are strictly wild.
The theorem combines \cite[Theorem~1]{j81} and a special case of \cite[Theorem 1]{cmt02}.

\begin{thm}\label{thm:rowrem}
Let $\la$ and $\mu$ be two partitions of $n$ with $\la_1=\mu_1$.
Define $\bar\la$ and $\bar\mu$ to be the partitions of $n-\la_1$ obtained from $\la$ and $\mu$ by removing the first row.
Then $d_{\la\mu}^{e,0}(v) = d_{\bar\la\bar\mu}^{e,0}(v)$ and $d_{\la\mu}^{e,p}(1) = d_{\bar\la\bar\mu}^{e,p}(1)$.
\end{thm}

\section{Reduction theorem}\label{sec:reduction}

The following lemma is the key first step in strengthening the results of \cite{als23}.

\begin{lemc}{simsonskow3}{Chapter XVIII, Corollary 4.7}\label{lem:ssstrictly}
Let $A$ be the path algebra $A=\bbf Q$ of a finite, acyclic, connected quiver $Q$.
Then $A$ is strictly wild if and only if $Q$ is not of finite or affine type $ADE$.
\end{lemc}

For us, it will suffice to apply the above lemma in the case where $Q$ is an `over-extended Dynkin diagram' of the form $A_3^{(1)\wedge}$ or $D_4^{(1)\wedge}$.
These are obtained from the affine Dynkin diagrams $A_3^{(1)}$ and $D_4^{(1)}$, respectively, by adding an extra node and connecting it to any existing node besides the quadrivalent node of $D_4^{(1)}$ by an edge.
More generally, over-extended Dynkin diagrams appear in the literature as the Dynkin diagrams for hyperbolic Kac--Moody algebras.
See~\cite{hyperbolicDynkin} for a classification of such diagrams, from which we take our notation.

\begin{cor}\label{cor:subquiver}
If the Gabriel quiver of a finite-dimensional algebra $A$ over $\bbf$ contains the quiver of an over-extended Dynkin diagram with zigzag orientation (i.e.~such that every vertex is a sink or a source) as a
subquiver, then A is strictly wild, and therefore brick-wild.
\end{cor}

\begin{proof}
Let $B$ denote the $\bbf$-algebra whose quiver is the over-extended Dynkin diagram with zigzag orientation.
Then $B$ is strictly wild by \cref{lem:ssstrictly}.
Moreover, $B$ is isomorphic to a quotient of $A$, so that there is a surjective homomorphism of algebras $A \rightarrow B$, and therefore there is a fully faithful exact functor $B\text{-}\mathrm{mod} \rightarrow A\text{-}\mathrm{mod}$ preserving indecomposability and isomorphism classes.
Precomposing with a fully faithful exact functor $\bbf \langle X, Y \rangle\lmod \rightarrow B\lmod$ which preserves indecomposability and isomorphism classes yields such a functor $\bbf \langle X, Y \rangle\lmod \rightarrow A\lmod$.
\end{proof}

\begin{prop}\label{prop:matrixtrick}
Suppose that $e\geq3$ and $\bbf$ has characteristic $p\geq 0$.
If a submatrix of the graded decomposition matrix in characteristic $0$ is of one of the following forms, and $d_{\la\mu}^{e,p}(1) = d_{\la\mu}^{e,0}(1)$ for all $e$-regular partitions $\la, \mu$ that label rows of the submatrix, then the block in which those partitions belong is strictly wild, and therefore brick-wild.
\begin{multicols}{2}
\begin{equation*}\label{targetmatrix}\tag{\(\dag\)}
\begin{pmatrix}
1&\cdot&\cdot&\cdot&\cdot\\
v&1&\cdot&\cdot&\cdot\\
*&v&1&\cdot&\cdot\\
*&v&*&1&\cdot\\
*&*&v&v&1
\end{pmatrix}
\end{equation*}
\begin{equation*}\label{targetmatrix1.5}\tag{\(\dag'\)}
\begin{pmatrix}
1&\cdot&\cdot&\cdot&\cdot\\
*&1&\cdot&\cdot&\cdot\\
v&v&1&\cdot&\cdot\\
*&v&*&1&\cdot\\
*&*&v&v&1
\end{pmatrix}
\end{equation*}
\begin{equation*}\label{targetmatrix2}\tag{\(\dag''\)}
\begin{pmatrix}
1&\cdot&\cdot&\cdot&\cdot\\
*&1&\cdot&\cdot&\cdot\\
*&v&1&\cdot&\cdot\\
v&*&v&1&\cdot\\
v&*&v&*&1
\end{pmatrix}
\end{equation*}
\begin{equation*}\label{targetmatrix3}\tag{\(\ddag\)}
\begin{pmatrix}
1&\cdot&\cdot&\cdot&\cdot\\
*&1&\cdot&\cdot&\cdot\\
v&v&1&\cdot&\cdot\\
*&*&v&1&\cdot\\
v&*&*&v&1
\end{pmatrix}
\end{equation*}
\begin{equation*}\label{targetmatrixalt}\tag{\(\ddag'\)}
\begin{pmatrix}
1&\cdot&\cdot&\cdot&\cdot\\
v&1&\cdot&\cdot&\cdot\\
*&v&1&\cdot&\cdot\\
v&*&v&1&\cdot\\
*&*&*&v&1
\end{pmatrix}
\end{equation*}
\begin{equation*}\label{targetmatrixalt2}\tag{\(\ddag''\)}
\begin{pmatrix}
1&\cdot&\cdot&\cdot&\cdot\\
v&1&\cdot&\cdot&\cdot\\
*&v&1&\cdot&\cdot\\
*&v&*&1&\cdot\\
v&*&*&v&1
\end{pmatrix}
\end{equation*}
\begin{equation*}\label{targetmatrixaltsquare}\tag{\(\clubsuit\)}
\begin{pmatrix}
1&\cdot&\cdot&\cdot&\cdot\\
*&1&\cdot&\cdot&\cdot\\
v&v&1&\cdot&\cdot\\
*&*&v&1&\cdot\\
*&v&*&v&1
\end{pmatrix}
\end{equation*}
\begin{equation*}\label{targetmatrixaltsquare1.5}\tag{\(\clubsuit'\)}
\begin{pmatrix}
1&\cdot&\cdot&\cdot&\cdot\\
v&1&\cdot&\cdot&\cdot\\
*&v&1&\cdot&\cdot\\
v&*&v&1&\cdot\\
*&*&v&*&1
\end{pmatrix}
\end{equation*}
\begin{equation*}\label{targetmatrixaltsquare2}\tag{\(\clubsuit''\)}
\begin{pmatrix}
1&\cdot&\cdot&\cdot&\cdot\\
v&1&\cdot&\cdot&\cdot\\
*&v&1&\cdot&\cdot\\
*&*&v&1&\cdot\\
*&v&*&v&1
\end{pmatrix}
\end{equation*}
\begin{equation*}\label{targetmatrixstar}\tag{\(\spadesuit\)}
\begin{pmatrix}
1&\cdot&\cdot&\cdot&\cdot\\
\cdot&1&\cdot&\cdot&\cdot\\
v&v&1&\cdot&\cdot\\
\cdot&v^2&v&1&\cdot\\
v^2&\cdot&v&\cdot&1\\
\cdot&\cdot&v^2&\cdot&v&1
\end{pmatrix}
\end{equation*}
\end{multicols}
\end{prop}

\begin{proof}
Both the statement and its proof are almost identical to \cite[Proposition~2.15]{als23}, with an extra vertex in the underlying subquiver of the Gabriel quiver, yielding $A_3^{(1)\wedge}$ or $D_4^{(1)\wedge}$, and applying \cref{cor:subquiver}.
\end{proof}

\section{Weight 2 blocks}\label{sec:wt2}

In this section, we will prove \cref{thm:MAIN} for weight 2 blocks.
In particular, for $e\geq 3$, we show that we can always find $A_3^{(1)\wedge}$ (i.e.~a square with an extra edge from one corner to a fifth vertex) as a subquiver of the Gabriel quiver of any weight 2 block, except for the (Scopes class of the) Rouquier block when $e=3$.
We say more about this block at the end of \cref{subsec:wt2fifthcase}.

First, we recall the necessary results on decomposition matrices and adjustment matrices.
Weight 2 blocks have trivial adjustment matrices for $p\geq 3$, by \cite{richardswt2}, so that our main focus is on $p=0$ or $2$.
For weight 2 blocks, we use the following notation for the partitions in the block.
First, we number the runners from $0$  to $e-1$, so that runner $i$ contains the marked position $p_i$, for $p_0 < p_1 < \dots < p_{e-1}$ as in \cref{sec:background}.
If the abacus display for $\la$ is obtained from that of its $e$-core by sliding the lowest beads on runners $i$ and $j$ each down one spot, with $i<j$, we denote the partition by $\langle i, j \rangle$.
If it is obtained by sliding the lowest bead on runner $i$ down two spaces, we denote it by $\langle i \rangle$.
Finally, if it is obtained by sliding each of the bottom two beads on runner $i$ down one space, we denote it $\langle i^2 \rangle$.

\begin{thmc}{faywt2}{Corollary 2.4}\label{thm:wt2adjust}
Let $\nu$ and $\mu$ be $e$-regular partitions in a weight $2$ block of $\hhh$, and let $p=2$.
Then
\[
a_{\nu\mu}(v) = 
\begin{cases}
1 &\text{ if } \nu = \langle i^2 \rangle, \ \mu =\langle i\rangle, \ p_i-p_{i-1}>e \text{ and } p_{i+1}-p_i>e \text{ for } 1\leq i \leq e-1;\\
1 &\text{ if } \nu = \langle i^2 \rangle, \ \mu =\langle i, i+1\rangle, \ p_i-p_{i-1}>e \text{ and } p_{i+1}-p_i<e \text{ for } 1\leq i < e-1;\\
\delta_{\nu\mu} &\text{ otherwise.}\\
\end{cases}
\]
\end{thmc}

As in \cite{als23}, we say that a partition $\la$ is \emph{adjacent} to an $e$-regular partition $\mu$ if $d_{\la\mu}^{e,0}(v) = v$.

\begin{thmc}{faywt2}{Theorem 3.2}\label{thm:wt2ext}
Suppose $p=2$, $B$ is a weight two block and $\la$, $\mu$ are $e$-regular partitions in $B$.
Then:
\begin{itemize}
\item If neither of $\la$ and $\mu$ is of the form $\langle i^2 \rangle$ for some $i$ such that $p_i - p_{i-1} > e$,
then $\Ext^1(\D\la, \D\mu) \cong \bbf$ if $\la$ and $\mu$ are adjacent, and is trivial otherwise.

\item Suppose $\la = \langle i^2 \rangle$ for some $i$ such that $p_i - p_{i-1} >e$.
Then $\Ext^1(\D\la, \D\mu) \cong \bbf$ if $\mu = \langle i \rangle$ with 
$p_{i+1} - p_i > e$
or if $\mu = \langle i, i+1 \rangle$ with 
$p_{i+1} - p_i < e$,
 and is trivial otherwise.
\end{itemize}
\end{thmc}

\subsection{$p_{e-1}-p_{e-3}<e$}\label{subsec:wt2firstcase}

This case was treated in \cite[Section~4.1]{als23}, and we used the partitions $\langle e-1 \rangle$, $\langle e-2 \rangle$, $\langle e-3 \rangle$, $\langle e-3,e-1 \rangle$ to yield our result there (obtaining a copy of $A^{(1)}_3$ in the Gabriel quiver).
Now, if we furthermore include $\langle e-2,e-1 \rangle$, and note that \cref{thm:runnerrem} allows us to restrict the problem of finding characteristic $0$ graded decomposition numbers down to $n=6$, as in loc.~cit., where $\langle e-2,e-1 \rangle = (3^2)$, and we find that the submatrix of the graded decomposition matrix given by these 5 partitions is of the form~(\ref{targetmatrix}).
\begin{answer}
\[
\begin{array}{r|ccccc}
&\rt{6}
&\rt{5,1}
&\rt{4,1^2}
&\rt{3^2}
&\rt{\mathrlap{3,2,1}\hphantom{\langle e-3,e-1 \rangle}}
\\\hline
6&1&\cdot&\cdot&\cdot&\cdot\\
5,1&v&1&\cdot&\cdot&\cdot\\
4,1^2&\cdot&v&1&\cdot&\cdot\\
3^2&\cdot&v&\cdot&1&\cdot\\
3,2,1&v&v^2&v&v&1
\end{array}
\qquad = \qquad
\begin{array}{r|ccccc}
&\rt{\langle e-1 \rangle}
&\rt{\langle e-2 \rangle}
&\rt{\langle e-3 \rangle}
&\rt{\langle e-2,e-1 \rangle}
&\rt{\langle e-3,e-1 \rangle}
\\\hline
\langle e-1 \rangle&1&\cdot&\cdot&\cdot&\cdot\\
\langle e-2 \rangle&v&1&\cdot&\cdot&\cdot\\
\langle e-3 \rangle&\cdot&v&1&\cdot&\cdot\\
\langle e-2,e-1 \rangle&\cdot&v&\cdot&1&\cdot\\
\langle e-3,e-1 \rangle&v&v^2&v&v&1
\end{array}
\]
\end{answer}
By \cref{thm:wt2adjust}, we see that this matrix is characteristic-free, and we thus obtain that the Gabriel quiver for any block with $p_{e-1}-p_{e-3}<e$ contains a copy of $A_3^{(1)\wedge}$, and by \cref{prop:matrixtrick} these blocks are therefore strictly wild and brick-wild.

\subsection{$p_{e-1}-p_{e-2}<e$ and $p_{e-2}-p_{e-3}<e$, but $p_{e-1}-p_{e-3}>e$}\label{subsec:wt2secondcase}

This case was treated in \cite[Section~4.2]{als23}, and we used the partitions $\langle e-1 \rangle$, $\langle e-2 \rangle$, $\langle e-2,e-1 \rangle$, $\langle (e-1)^2 \rangle$ to obtain a copy of $A^{(1)}_3$ in the Gabriel quiver.
Similarly to \cref{subsec:wt2firstcase}, we include the additional partition $\langle e-3 \rangle$, and we use \cref{thm:runnerrem} to reduce the problem to the corresponding $n=7$ block, where we obtain that the graded decomposition matrix is of the form~(\ref{targetmatrix2}).
\begin{answer}
\[
\begin{array}{r|ccccc}
&\rt{7}
&\rt{5,2}
&\rt{4,3}
&\rt{4,2,1}
&\rt{\mathrlap{3,2,1^2}\hphantom{\langle e-2,e-1 \rangle}}
\\\hline
7&1&\cdot&\cdot&\cdot&\cdot\\
5,2&v&1&\cdot&\cdot&\cdot\\
4,3&\cdot&v&1&\cdot&\cdot\\
4,2,1&v&v^2&v&1&\cdot\\
3,2,1^2&v&\cdot&v&v^2&1
\end{array}
\qquad = \qquad
\begin{array}{r|ccccc}
&\rt{\langle e-1 \rangle}
&\rt{\langle e-2 \rangle}
&\rt{\langle e-2,e-1 \rangle}
&\rt{\langle (e-1)^2 \rangle}
&\rt{\langle e-3 \rangle}
\\\hline
\langle e-1 \rangle&1&\cdot&\cdot&\cdot&\cdot\\
\langle e-2 \rangle&v&1&\cdot&\cdot&\cdot\\
\langle e-2,e-1 \rangle&\cdot&v&1&\cdot&\cdot\\
\langle (e-1)^2 \rangle&v&v^2&v&1&\cdot\\
\langle e-3 \rangle&v&\cdot&v&v^2&1
\end{array}
\]
\end{answer}
This matrix is again characteristic-free, by \cref{thm:wt2adjust}.
By \cref{prop:matrixtrick}, any block with $p_{e-1}-p_{e-2}<e$ and $p_{e-2}-p_{e-3}<e$, but $p_{e-1}-p_{e-3}>e$ is therefore strictly wild and brick-wild.

\subsection{$p_{e-1} - p_{e-2} > e$ and $p_{e-2} - p_{e-3} < e$}\label{subsec:wt2thirdcase}

This case was treated in \cite[Section~4.3]{als23}, and we used the partitions $\langle e-1 \rangle$, $\langle e-2,e-1 \rangle$, $\langle e-2 \rangle$, $\langle e-3 \rangle$ to obtain a copy of $A^{(1)}_3$ in the Gabriel quiver.
We include the additional partition $\langle (e-1)^2 \rangle$, and again use \cref{thm:runnerrem} to reduce the problem to the corresponding $n=8$ block, where we obtain that the graded decomposition matrix in characteristic $0$ is of the form~(\ref{targetmatrix3}).
\begin{answer}
\[
\begin{array}{r|ccccc}
&\rt{8}
&\rt{5,3}
&\rt{5,2,1}
&\rt{4,3,1}
&\rt{\mathrlap{3^2,1^2}\hphantom{\langle e-2,e-1 \rangle}}
\\\hline
8&1&\cdot&\cdot&\cdot&\cdot\\
5,3&\cdot&1&\cdot&\cdot&\cdot\\
5,2,1&v&v&1&\cdot&\cdot\\
4,3,1&\cdot&v^2&v&1&\cdot\\
3^2,1^2&v&\cdot&v^2&v&1
\end{array}
\qquad = \qquad
\begin{array}{r|ccccc}
&\rt{\langle e-1 \rangle}
&\rt{\langle (e-1)^2 \rangle}
&\rt{\langle e-2,e-1 \rangle}
&\rt{\langle e-2 \rangle}
&\rt{\langle e-3 \rangle}
\\\hline
\langle e-1 \rangle&1&\cdot&\cdot&\cdot&\cdot\\
\langle (e-1)^2 \rangle&\cdot&1&\cdot&\cdot&\cdot\\
\langle e-2,e-1 \rangle&v&v&1&\cdot&\cdot\\
\langle e-2 \rangle&\cdot&v^2&v&1&\cdot\\
\langle e-3 \rangle&v&\cdot&v^2&v&1
\end{array}
\]
\end{answer}
Now, these graded decomposition numbers are not characteristic-free, but we may instead appeal to \cref{thm:wt2ext} to note that in characteristic 2 there is a non-trivial extension in $\Ext^1(\D{\langle e-1 \rangle}, \D{\langle (e-1)^2 \rangle})$, which combined with the fact that $\langle e-1 \rangle$, $\langle e-2,e-1 \rangle$, $\langle e-2 \rangle$, $\langle e-3 \rangle$ give a copy of $A^{(1)}_3$, tells us that we have a copy of $A^{(1)\wedge}_3$ in the Gabriel quiver, so that \cref{prop:matrixtrick} again tells us that these blocks are strictly wild and brick-wild.
Note that in characteristic $p\neq 2$, $\Ext^1(\D{\langle e-1 \rangle}, \D{\langle (e-1)^2 \rangle}) = 0$, and it is actually an extension in $\Ext^1(\D{\langle e-2,e-1 \rangle}, \D{\langle (e-1)^2 \rangle})$ that our matrix (or an application of \cite[Theorem~3.1]{faywt2}, based on \cite[Theorem~6.1]{ct01}) yields, but in either case we have the necessary copy of $A^{(1)\wedge}_3$.

\subsection{$p_{e-1} - p_{e-2} < e$ and $p_{e-2} - p_{e-3} > e$}\label{subsec:wt2fourthcase}

Next, suppose that $\rho$ satisfies $p_{e-1} - p_{e-2} < e$ and $p_{e-2} - p_{e-3} > e$.

This case was treated in \cite[Section~4.4]{als23}, and we used the partitions $\langle e-1 \rangle$, $\langle e-2 \rangle$, $\langle e-2,e-1 \rangle$, $\langle (e-1)^2 \rangle$ to obtain a copy of $A^{(1)}_3$ in the Gabriel quiver.
We include the additional partition $\langle (e-2)^2 \rangle$, and again use \cref{thm:runnerrem} to reduce the problem to the corresponding $n=8$ block, where the graded decomposition matrix in characteristic $0$ is of the form~(\ref{targetmatrixalt}), so that the result follows if the characteristic is not 2.
\begin{answer}
\[
\begin{array}{r|ccccc}
&\rt{7,1}
&\rt{6,2}
&\rt{4^2}
&\rt{4,2^2}
&\rt{\mathrlap{3,2^2,1}\hphantom{\langle e-2,e-1 \rangle}}
\\\hline
7,1&1&\cdot&\cdot&\cdot&\cdot\\
6,2&v&1&\cdot&\cdot&\cdot\\
4^2&\cdot&v&1&\cdot&\cdot\\
4,2^2&v&v^2&v&1&\cdot\\
3,2^2,1&v^2&\cdot&\cdot&v&1
\end{array}
\qquad = \qquad
\begin{array}{r|ccccc}
&\rt{\langle e-1 \rangle}
&\rt{\langle e-2 \rangle}
&\rt{\langle e-2,e-1 \rangle}
&\rt{\langle (e-1)^2 \rangle}
&\rt{\langle (e-2)^2 \rangle}
\\\hline
\langle e-1 \rangle&1&\cdot&\cdot&\cdot&\cdot\\
\langle e-2 \rangle&v&1&\cdot&\cdot&\cdot\\
\langle e-2,e-1 \rangle&\cdot&v&1&\cdot&\cdot\\
\langle (e-1)^2 \rangle&v&v^2&v&1&\cdot\\
\langle (e-2)^2 \rangle&v^2&\cdot&\cdot&v&1
\end{array}
\]
\end{answer}
Now, as in the previous case, these graded decomposition numbers are not characteristic-free, but we may instead appeal to \cref{thm:wt2ext} to note that in characteristic 2, there is a non-trivial extension in $\Ext^1(\D{\langle e-2,e-1 \rangle}, \D{\langle (e-2)^2 \rangle})$ (rather than the extension in $\Ext^1(\D{\langle (e-1)^2 \rangle}, \D{\langle (e-2)^2 \rangle})$ that exists in other characteristics).
Combining this with the fact that $\langle e-1 \rangle$, $\langle e-2 \rangle$, $\langle e-2,e-1 \rangle$, $\langle (e-1)^2 \rangle$ give a copy of $A^{(1)}_3$, tells us that we have a copy of $A^{(1)\wedge}_3$ in the Gabriel quiver, so that \cref{prop:matrixtrick} again tells us that these blocks are strictly wild and brick-wild.

\subsection{$p_{e-1} - p_{e-2} > e$ and $p_{e-2} - p_{e-3} > e$}\label{subsec:wt2fifthcase}

Finally, suppose that $\rho$ satisfies $p_{e-1} - p_{e-2} > e$ and $p_{e-2} - p_{e-3} > e$.
This case was treated in \cite[Section~4.5]{als23}, and we used the partitions $\langle e-1 \rangle$, $\langle (e-1)^2 \rangle$, $\langle e-2,e-1 \rangle$, $\langle e-2 \rangle$, and $\langle (e-2)^2 \rangle$ to obtain a copy of $D^{(1)}_4$ in the Gabriel quiver whenever $p\neq 2$.
If $e=3$, this is in fact the entire Gabriel quiver, and this case is just a single Scopes class, the class containing the Rouquier block.
\cref{prop:matrixtrick} does not apply to this block, and we are unable to determine whether or not it is strictly wild.

We assume for now that $e\geq 4$.
Taking partitions 
$\langle (e-1)^2 \rangle$, $\langle e-2,e-1 \rangle$, $\langle e-3,e-1 \rangle$, $\langle (e-2)^2 \rangle$, and $\langle e-3,e-2 \rangle$,
we may apply \cref{thm:runnerrem} to reduce the problem to the corresponding $n=17$ or $23$ block, depending on whether $p_{e-3} - p_{e-4} < e$ or $p_{e-3} - p_{e-4} > e$, respectively.
In either case, we may compute the graded decomposition numbers in characteristic $0$ by the LLT algorithm~\cite{LLT}.
By \cref{thm:wt2adjust}, since we have included no partition of the form $\langle i \rangle$, and we are in the case where $p_{e-1} - p_{e-2} > e$
 and $p_{e-2} - p_{e-3} > e$,
  our submatrix is characteristic-free, yielding matrices of the form~(\ref{targetmatrix}) in any characteristic.
\begin{align*}
\begin{array}{r|ccccc}
&\rt{9,6,2}
&\rt{9,5,3}
&\rt{9,2^3,1^2}
&\rt{5^2,3^2,1}
&\rt{\mathrlap{5^2,3,2,1^2}\hphantom{\langle e-3,e-1 \rangle}}
\\\hline
9,6,2&1&\cdot&\cdot&\cdot&\cdot\\
9,5,3&v&1&\cdot&\cdot&\cdot\\
9,2^3,1^2&\cdot&v&1&\cdot&\cdot\\
5^2,3^2,1&\cdot&v&\cdot&1&\cdot\\
5^2,3,2,1^2&\cdot&v^2&v&v&1
\end{array}
\qquad &= \qquad
\begin{array}{r|ccccc}
&\rt{\langle (e-1)^2 \rangle}
&\rt{\langle e-2,e-1 \rangle}
&\rt{\langle e-3,e-1 \rangle}
&\rt{\langle (e-2)^2 \rangle}
&\rt{\langle e-3,e-2 \rangle}
\\\hline
\langle (e-1)^2 \rangle&1&\cdot&\cdot&\cdot&\cdot\\
\langle e-2,e-1 \rangle&v&1&\cdot&\cdot&\cdot\\
\langle e-3,e-1 \rangle&\cdot&v&1&\cdot&\cdot\\
\langle (e-2)^2 \rangle&\cdot&v&\cdot&1&\cdot\\
\langle e-3,e-2 \rangle&\cdot&v^2&v&v&1
\end{array}\\
\\
\qquad &= \qquad
\begin{array}{r|ccccc}
&\rt{10,7,3,1^3}
&\rt{10,6,4,1^3}
&\rt{10,3^3,2^2}
&\rt{6^2,4^2,2,1}
&\rt{6^2,4,3,2^2}
\\\hline
10,7,3,1^3&1&\cdot&\cdot&\cdot&\cdot\\
10,6,4,1^3&v&1&\cdot&\cdot&\cdot\\
10,3^3,2^2&\cdot&v&1&\cdot&\cdot\\
6^2,4^2,2,1&\cdot&v&\cdot&1&\cdot\\
6^2,4,3,2^2&\cdot&v^2&v&v&1
\end{array}
\end{align*}
We thus obtain that for $e\geq4$, any block with $p_{e-1}-p_{e-2}>e$ and $p_{e-2}-p_{e-3}>e$ is strictly wild and brick-wild, by \cref{prop:matrixtrick}.

Finally, we return to the $e=3$ case, in which only the Scopes class containing the Rouquier block  satisfies $p_{e-1} - p_{e-2} > e$ and $p_{e-2} - p_{e-3} > e$.
For this block, if $p\neq 2$, the Gabriel quiver is $D^{(1)}_4$, while if $p=2$, the quiver is $A_5$.
In either case, our methods are unable to determine that these blocks are strictly wild.
We invite experts to attempt to resolve this final block.

\begin{answer}
From the decomposition matrices and some computations of homomorphisms between Specht modules, we can determine that for $p\neq2$, the five projective indecomposable modules in the block are as follows, where $3$ labels the vertex in the middle of the star $D^{(1)}_4$.
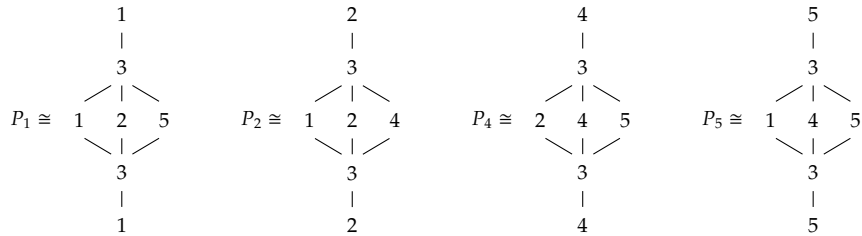
\[
\begin{tikzpicture}[scale=0.7]{\scalefont{0.7}
\draw (0.3, 0) node {$P_1 \cong$};
\draw (2,2) node {1};
\draw (2,1) node {3};
\draw (2,0) node {2};
\draw (1.2,0) node {1};
\draw (2.8,0) node {5};
\draw (2,-1) node {3};
\draw (2,-2) node {1};
}
\draw (2,1.65)--(2,1.35);
\draw (2,0.65)--(2,0.35);
\draw (2,-0.35)--(2,-0.65);
\draw (2,-1.35)--(2,-1.65);

\draw (1.8,0.65)--(1.3,0.35);
\draw (2.2,0.65)--(2.7,0.35);

\draw (1.3,-0.35)--(1.8,-0.65);
\draw (2.7,-0.35)--(2.2,-0.65);
\end{tikzpicture}
\qquad
\begin{tikzpicture}[scale=0.7]{\scalefont{0.7}
\draw (0.3, 0) node {$P_2 \cong$};
\draw (2,2) node {2};
\draw (2,1) node {3};
\draw (2,0) node {2};
\draw (1.2,0) node {1};
\draw (2.8,0) node {4};
\draw (2,-1) node {3};
\draw (2,-2) node {2};
}
\draw (2,1.65)--(2,1.35);
\draw (2,0.65)--(2,0.35);
\draw (2,-0.35)--(2,-0.65);
\draw (2,-1.35)--(2,-1.65);

\draw (1.8,0.65)--(1.3,0.35);
\draw (2.2,0.65)--(2.7,0.35);

\draw (1.3,-0.35)--(1.8,-0.65);
\draw (2.7,-0.35)--(2.2,-0.65);
\end{tikzpicture}
\qquad
\begin{tikzpicture}[scale=0.7]{\scalefont{0.7}
\draw (0.3, 0) node {$P_4 \cong$};
\draw (2,2) node {4};
\draw (2,1) node {3};
\draw (2,0) node {4};
\draw (1.2,0) node {2};
\draw (2.8,0) node {5};
\draw (2,-1) node {3};
\draw (2,-2) node {4};
}
\draw (2,1.65)--(2,1.35);
\draw (2,0.65)--(2,0.35);
\draw (2,-0.35)--(2,-0.65);
\draw (2,-1.35)--(2,-1.65);

\draw (1.8,0.65)--(1.3,0.35);
\draw (2.2,0.65)--(2.7,0.35);

\draw (1.3,-0.35)--(1.8,-0.65);
\draw (2.7,-0.35)--(2.2,-0.65);
\end{tikzpicture}
\qquad
\begin{tikzpicture}[scale=0.7]{\scalefont{0.7}
\draw (0.3, 0) node {$P_5 \cong$};
\draw (2,2) node {5};
\draw (2,1) node {3};
\draw (2,0) node {4};
\draw (1.2,0) node {1};
\draw (2.8,0) node {5};
\draw (2,-1) node {3};
\draw (2,-2) node {5};
}
\draw (2,1.65)--(2,1.35);
\draw (2,0.65)--(2,0.35);
\draw (2,-0.35)--(2,-0.65);
\draw (2,-1.35)--(2,-1.65);

\draw (1.8,0.65)--(1.3,0.35);
\draw (2.2,0.65)--(2.7,0.35);

\draw (1.3,-0.35)--(1.8,-0.65);
\draw (2.7,-0.35)--(2.2,-0.65);
\end{tikzpicture}
\]
The final projective indecomposable module, $P_3$, has the following radical structure.
\[
\begin{tikzpicture}[scale=0.7]{\scalefont{0.7}
\draw (2,2) node {3};
\draw (0.8,1) node {1};
\draw (1.6,1) node {2};
\draw (2.4,1) node {4};
\draw (3.2,1) node {5};
\draw (2,0) node {3};
\draw (1.2,0) node {3};
\draw (2.8,0) node {3};

\draw (0.8,-1) node {1};
\draw (1.6,-1) node {2};
\draw (2.4,-1) node {4};
\draw (3.2,-1) node {5};
%
\draw (2,-2) node {3};
}
%
%
\end{tikzpicture}
\]
\end{answer}

\section{Weight 3 blocks}\label{sec:wt3}

In this section, we will prove \cref{thm:MAIN} for weight 3 blocks.
In particular, for $e\geq 3$, we show that we can always find $A_3^{(1)\wedge}$ or $D_4^{(1)\wedge}$ (a four-pointed star with an extra edge from one point to a fifth vertex) as a subquiver of the Gabriel quiver for weight 3 blocks.

As in \cref{sec:wt2}, we must recall some notation for the partitions in a weight 3 block, once again using our ordering on runners with runner $i$ containing position $p_i$, with $p_0 < p_1 < \dots < p_{e-1}$.
In this section, $\langle i \rangle$ will denote the partition obtained from the core by sliding the lowest bead on runner $i$ down three places, while $\langle i, j \rangle$ will denote the partition obtained from the core by sliding the lowest bead on runner $i$ down two places, and the lowest bead on runner $j$ down one place.
Note that we allow $i=j$, so that $\langle i, i \rangle$ is obtained by sliding the lowest bead on runner $i$ down two places, and the second lowest bead down one place.
We let $\langle i^2, j \rangle$ denote the partition whose abacus display is obtained from that of the core by sliding down the lowest two beads on runner $i$ and the lowest bead on runner $j$ one place each.
We let $\langle i^3 \rangle$ denote the partition whose abacus display is obtained from that of the core by sliding the lowest three beads on runner $i$ down one place each.
Finally, we let $\langle i, j, k\rangle$, for $i<j<k$, denote the partition whose abacus display is obtained from that of the core by sliding down the lowest bead on runners $i$, $j$, and $k$ one place each.

As in \cite{als23}, we use results of Fayers and Tan to determine the adjustment matrices in characteristics 2 and 3, while they are known to be trivial in larger characteristics (\cite{fay08wt3}).
For the following result, the reader is invited to see \cite{faytan06} for the definition of a Rouquier block, and of inducing semi-simply and almost semi-simply -- we only require the explicit determination of when these occur.

\begin{thmc}{faytan06}{Theorem 3.3 and Propositions 3.4 and 3.5}\label{thm:wt3adj}
Suppose that $B$ is a weight $3$ block of $\hhh$.
\begin{enumerate}[label=(\roman*)]
\item If $\nchar \bbf = 2$, then
$a_{\nu\mu}(v) = 1$ if there is a Rouquier block $C$ and some $1\leq i\neq k \leq e-1$ such that
\begin{itemize}
\item $\nu$ induces semi-simply or almost semi-simply to $\langle i^3 \rangle$ in $C$, while $\mu$ induces semi-simply to $\langle i \rangle$ in $C$; or such that
\item $\nu$ induces semi-simply to $\langle i^2, k \rangle$ in $C$, while $\mu$ induces semi-simply to $\langle i, k \rangle$ in $C$.
\end{itemize}

\item If $\nchar \bbf = 3$, then
$a_{\nu\mu}(v) = 1$ if there is a Rouquier block $C$ and some $1\leq i \leq e-1$ such that
\begin{itemize}
\item $\nu$ induces semi-simply to $\langle i^3 \rangle$ in $C$, while $\mu$ induces semi-simply to $\langle i, i \rangle$ in $C$; or such that
\item $\nu$ induces semi-simply to $\langle i, i \rangle$ in $C$, while $\mu$ induces semi-simply to $\langle i \rangle$ in $C$.
\end{itemize}
\end{enumerate}
For all other $\nu$ and $\mu$, $a_{\nu\mu}(v) = \delta_{\nu\mu}$.

In particular, a partition $\la$ in $B$ induces up semi-simply to a partition $\omega$ in $C$ of one of the forms above if and only if $\omega$, $\la$ and $B$ satisfy one of the following sets of conditions, where $1 \leq i < k$.
(Note that as a matter of convention, we consider $p_e$ to be an arbitrary integer satisfying $p_e > p_{e-1} + 2e$.)

\FloatBarrier
\begin{table}[h]
\centering
\begin{tabular}{ccc}\toprule
$\omega$ & $\la$ & Conditions on $B$\\\midrule
& $\langle i \rangle$ & $p_{i+1} - p_i > 2e$\\\cmidrule{2-3}
$\langle i \rangle$ & $\langle i, i+1 \rangle$ & $p_{i+1} - p_i < 2e$, $p_{i+2} - p_i > e$\\\cmidrule{2-3}
& $\langle i, i+1, i+2 \rangle$ & $p_{i+2} - p_i < e$\\\midrule
\multirow{2}{*}{$\langle i, i \rangle$} & $\langle i, i \rangle$ & $p_{i+1} - p_i > e$, $p_i - p_{i-1} > e$\\\cmidrule{2-3}
& $\langle i^2, i+1 \rangle$ & $p_{i+1} - p_i < e$, $p_i - p_{i-1} > e$\\\midrule
$\langle i^3 \rangle$ & $\langle i^3 \rangle$ & $p_i - p_{i-1} > 2e$\\\midrule
& $\langle i, k \rangle$ & $p_k - p_i > 2e$, $p_{i+1} - p_i > e$\\\cmidrule{2-3}
$\langle i, k \rangle$ & $\langle i, i+1, k \rangle$ & $p_k - p_{i+1} > e$, $p_{i+1} - p_i < e$\\\cmidrule{2-3}
& $\langle k^2, i \rangle$ & $p_k - p_{i+1} < e$, $p_k - p_i < 2e$, $p_k - p_{i-1} > e$ \\\midrule
& $\langle k, i \rangle$ & $p_{k+1} - p_k > e$, $p_k - p_i > e$\\\cmidrule{2-3}
\multirow{2}{*}{$\langle k, i \rangle$} & $\langle k, k \rangle$ & $p_{k+1} - p_k > e$, $p_k - p_i < e$, $p_k - p_{i-1} > e$\\\cmidrule{2-3}
& $\langle i, k, k+1 \rangle$ & $p_{k+1} - p_k < e$, $p_k - p_i > e$\\\cmidrule{2-3}
& $\langle k^2, k+1 \rangle$ & $p_{k+1} - p_k < e$, $p_k - p_i < e$, $p_k - p_{i-1} > e$\\\midrule
\multirow{2}{*}{$\langle i^2, k \rangle$} & $\langle i^2, k \rangle$ & $p_k - p_i > e$, $p_i - p_{i-1} > e$\\\cmidrule{2-3}
& $\langle k^3 \rangle$ & $p_k - p_i < e$, $p_k - p_{i-1} > 2e$\\\midrule
\multirow{2}{*}{$\langle k^2, i \rangle$} & $\langle k^2, i \rangle$ & $p_k - p_i > 2e$, $p_k - p_{k-1} > e$\\\cmidrule{2-3}
& $\langle k^3 \rangle$ & $2e > p_k - p_i > e$, $p_k - p_{i-1} > 2e$, $2e > p_k - p_{k-1} > e$\\
\bottomrule
\end{tabular}
\end{table}
\FloatBarrier
Moreover, if $e< p_i - p_{i-1} < 2e$ and $p_{i+1} - p_{i-1} > 2e$, then $\langle i-1 \rangle$ induces almost semi-simply to $\langle i^3 \rangle$ in a Rouquier block $C$.
\end{thmc}

We will use the above result in much the same way as we did in \cite[Section~5]{als23}.

\subsection{$p_{e-1}-p_{e-3}<e$}\label{subsec:wt3firstcase}

This case was treated in \cite[Section~5.1]{als23}, and we used the partitions $\langle e-3 \rangle$, $\langle e-1, e-3 \rangle$, $\langle e-2, e-3 \rangle$, $\langle e-3,e-2 \rangle$ to yield our result there (obtaining a copy of $A^{(1)}_3$ in the Gabriel quiver).
Now, if we furthermore include $\langle e-1,e-2 \rangle$, and note that \cref{thm:runnerrem} allows us to restrict the problem of finding characteristic $0$ graded decomposition numbers down to $n=9$, where $\langle e-1,e-2 \rangle = (6,3)$, and we find that the submatrix of the graded decomposition matrix given by these 5 partitions is of the form~(\ref{targetmatrix3}).
\begin{answer}
\[
\begin{array}{r|ccccc}
&\rt{7,1^2}
&\rt{6,3}
&\rt{6,2,1}
&\rt{5,2^2}
&\rt{\mathrlap{4,3,2}\hphantom{\langle e-3,e-2 \rangle}}
\\\hline
7,1^2&1&\cdot&\cdot&\cdot&\cdot\\
6,3&\cdot&1&\cdot&\cdot&\cdot\\
6,2,1&v&v&1&\cdot&\cdot\\
5,2^2&\cdot&v^2&v&1&\cdot\\
4,3,2&v&v^3&v^2&v&1
\end{array}
\qquad = \qquad
\begin{array}{r|ccccc}
&\rt{\langle e-3 \rangle}
&\rt{\langle e-1,e-2 \rangle}
&\rt{\langle e-1,e-3 \rangle}
&\rt{\langle e-2,e-3 \rangle}
&\rt{\langle e-3,e-2 \rangle}
\\\hline
\langle e-3 \rangle&1&\cdot&\cdot&\cdot&\cdot\\
\langle e-1,e-2 \rangle&\cdot&1&\cdot&\cdot&\cdot\\
\langle e-1,e-3 \rangle&v&v&1&\cdot&\cdot\\
\langle e-2,e-3 \rangle&\cdot&v^2&v&1&\cdot\\
\langle e-3,e-2 \rangle&v&v^3&v^2&v&1
\end{array}
\]
\end{answer}
By \cref{thm:wt3adj}, we see that: $\langle e-1\rangle$ induces semi-simply to $\langle e-1\rangle$;
$\langle e-2, e-1\rangle$ induces semi-simply to $\langle e-2\rangle$;
no partition induces semi-simply to $\langle e-2, e-1\rangle$ or $\langle e-1, e-2\rangle$;
and no partition induces semi-simply to $\langle e-1, e-1\rangle$ or $\langle e-2, e-2\rangle$.
Since our chosen partitions do not include $\langle e-1\rangle$ or $\langle e-2, e-1\rangle$, $a_{\nu\mu} = \delta_{\nu\mu}$ for all $\mu$ among them, and we see that the above submatrix is characteristic-free.
We thus obtain that the Gabriel quiver for any block with $p_{e-1}-p_{e-3}<e$ contains a copy of $A_3^{(1)\wedge}$, so by \cref{prop:matrixtrick} these blocks are strictly wild and brick-wild.

\subsection{$p_{e-1} - p_{e-2} < e$ and $p_{e-2} - p_{e-3} < e$, but $p_{e-1} - p_{e-3} > e$}\label{subsec:wt3secondcase}

In \cite[Section~5.2]{als23}, we used the partitions $\langle e-2 \rangle$, $\langle e-1, e-2 \rangle$, $\langle e-3 \rangle$, and $\langle e-2, e-3 \rangle$ to obtain a copy of $A^{(1)}_3$ in the Gabriel quiver.
If we furthermore include the partition $\langle e-1, e-1 \rangle$, and apply \cref{thm:runnerrem}, we restrict the characteristic $0$ decomposition matrix problem down to $n=10$, where $\langle e-1, e-1 \rangle = (7,2,1)$, and we find that the submatrix of the graded decomposition matrix given by these 5 partitions is of the form~(\ref{targetmatrixalt2}).
\begin{answer}
\[
\begin{array}{r|ccccc}
&\rt{8,2}
&\rt{7,3}
&\rt{7,2,1}
&\rt{6,2,1^2}
&\rt{\mathrlap{5,2^2,1}\hphantom{\langle e-3,e-2 \rangle}}
\\\hline
8,2&1&\cdot&\cdot&\cdot&\cdot\\
7,3&v&1&\cdot&\cdot&\cdot\\
7,2,1&v^2&v&1&\cdot&\cdot\\
6,2,1^2&\cdot&v&v^2&1&\cdot\\
5,2^2,1&v&v^2&\cdot&v&1
\end{array}
\qquad = \qquad
\begin{array}{r|ccccc}
&\rt{\langle e-2 \rangle}
&\rt{\langle e-1,e-2 \rangle}
&\rt{\langle e-1,e-1 \rangle}
&\rt{\langle e-3 \rangle}
&\rt{\langle e-2,e-3 \rangle}
\\\hline
\langle e-2 \rangle&1&\cdot&\cdot&\cdot&\cdot\\
\langle e-1,e-2 \rangle&v&1&\cdot&\cdot&\cdot\\
\langle e-1,e-1 \rangle&v^2&v&1&\cdot&\cdot\\
\langle e-3 \rangle&\cdot&v&v^2&1&\cdot\\
\langle e-2,e-3 \rangle&v&v^2&\cdot&v&1
\end{array}
\]
\end{answer}
We apply \cref{thm:wt3adj} again.
We see that: $\langle e-1\rangle$ induces semi-simply to $\langle e-1\rangle$;
$\langle e-2, e-1\rangle$ induces semi-simply to $\langle e-2\rangle$;
$\langle (e-1)^2, e-2\rangle$ induces semi-simply to $\langle e-2, e-1\rangle$, but no partition induces semi-simply to $\langle (e-2)^2, e-1\rangle$;
$\langle e-1, e-1\rangle$ induces semi-simply to $\langle e-1, e-2\rangle$ but no partition induces semi-simply to $\langle (e-1)^2, e-2\rangle$;
and no partition induces semi-simply to $\langle e-1, e-1\rangle$ or $\langle e-2, e-2\rangle$.
Since our chosen partitions do not include $\langle e-1\rangle$ or $\langle e-2, e-1\rangle$, $a_{\nu\mu} = \delta_{\nu\mu}$ for all $\mu$ among them, and we see that the above submatrix is characteristic-free.
By \cref{prop:matrixtrick}, these blocks are therefore strictly wild and brick-wild.

\subsection{$p_{e-1} - p_{e-2} > e$ and $p_{e-2} - p_{e-3} < e$}\label{subsec:wt3thirdcase}

As in \cite[Section~5.3]{als23}, the $e=3$ situation reduces to three Scopes classes, with minimal representatives being the blocks with cores $\rho^{(1)}=(2)$, $\rho^{(2)}=(3,1)$, and $\rho^{(3)}=(4,2)$, respectively.
We handle these one by one below, and it is clear that we may apply \cref{thm:runnerrem} to reduce from arbitrary $e$ down to $e=3$, landing in one of these cases.

\begin{enumerate}
\item
The weight 3 block with core $\rho^{(1)}=(2)$.
In \cite[Section~5.3]{als23}, we used the partitions $\langle e-1,e-1 \rangle$, $\langle e-1, e-2 \rangle$, $\langle (e-1)^2, e-2 \rangle$, and $\langle (e-1)^2,e-3 \rangle$.
Here we also include $\langle e-2 \rangle = (7,3,1)$, yielding a submatrix of the graded decomposition matrix of the form~(\ref{targetmatrixalt2}).
\begin{answer}
\[
\begin{array}{r|ccccc}
&\rt{8,3}
&\rt{8,2,1}
&\rt{7,3,1}
&\rt{5,3^2}
&\rt{\mathrlap{5,3,2,1}\hphantom{\langle (e-1)^2,e-3 \rangle}}
\\\hline
8,3&1&\cdot&\cdot&\cdot&\cdot\\
8,2,1&v&1&\cdot&\cdot&\cdot\\
7,3,1&v^2&v&1&\cdot&\cdot\\
5,3^2&\cdot&v&v^2&1&\cdot\\
5,3,2,1&v&v^2&v^3&v&1
\end{array}
\qquad = \qquad
\begin{array}{r|ccccc}
&\rt{\langle e-1,e-1 \rangle}
&\rt{\langle e-1,e-2 \rangle}
&\rt{\langle e-2 \rangle}
&\rt{\langle (e-1)^2,e-2 \rangle}
&\rt{\langle (e-1)^2,e-3 \rangle}
\\\hline
\langle e-1,e-1 \rangle&1&\cdot&\cdot&\cdot&\cdot\\
\langle e-1,e-2 \rangle&v&1&\cdot&\cdot&\cdot\\
\langle e-2 \rangle&v^2&v&1&\cdot&\cdot\\
\langle (e-1)^2,e-2 \rangle&\cdot&v&v^2&1&\cdot\\
\langle (e-1)^2,e-3 \rangle&v&v^2&v^3&v&1
\end{array}
\]
\end{answer}
\cref{thm:runnerrem} allows us, as usual, to reduce to the above case in characteristic $0$.
As in \cref{subsec:wt3secondcase}, we apply \cref{thm:wt3adj} and see that:
$\langle e-1\rangle$ induces semi-simply to $\langle e-1\rangle$;
$\langle e-2, e-1\rangle$ induces semi-simply to $\langle e-2\rangle$;
$\langle (e-1)^2, e-2\rangle$ induces semi-simply to $\langle e-2, e-1\rangle$, but no partition induces semi-simply to $\langle (e-2)^2, e-1\rangle$;
$\langle e-1, e-2\rangle$ induces semi-simply to $\langle e-1, e-2\rangle$ but no partition induces semi-simply to $\langle (e-1)^2, e-2\rangle$;
$\langle e-1, e-1\rangle$ induces semi-simply to $\langle e-1, e-1\rangle$ but no partition induces semi-simply to $\langle (e-1)^3\rangle$;
and no partition induces semi-simply to $\langle e-2, e-2\rangle$.
Since our chosen partitions do not include $\langle e-1\rangle$ or $\langle e-2, e-1\rangle$, $a_{\nu\mu} = \delta_{\nu\mu}$ for all $\mu$ among them, and we see that the above submatrix is characteristic-free.
By \cref{prop:matrixtrick}, these blocks are therefore strictly wild and brick-wild.

\item The weight 3 block with core $\rho^{(2)}=(3,1)$.
In \cite[Section~5.3]{als23}, we used the partitions $\langle (e-1)^2, e-2 \rangle$, $\langle e-3, e-1 \rangle$, $\langle e-3 \rangle$, and $\langle e-2, e-3 \rangle$.
Here we also include $\langle e-1,e-1 \rangle = (9,4)$, yielding a submatrix of the graded decomposition matrix of the form~(\ref{targetmatrixaltsquare}).
\begin{answer}
\[
\begin{array}{r|ccccc}
&\rt{9,4}
&\rt{6,4,3}
&\rt{6,3,2,1^2}
&\rt{5,4,2,1^2}
&\rt{\mathrlap{4^2,2^2,1}\hphantom{\langle (e-1)^2,e-2 \rangle}}
\\\hline
9,4&1&\cdot&\cdot&\cdot&\cdot\\
6,4,3&\cdot&1&\cdot&\cdot&\cdot\\
6,3,2,1^2&v&v&1&\cdot&\cdot\\
5,4,2,1^2&v^2&v^2&v&1&\cdot\\
4^2,2^2,1&v^3&v&v^2&v&1
\end{array}
\qquad = \qquad
\begin{array}{r|ccccc}
&\rt{\langle e-1,e-1 \rangle}
&\rt{\langle (e-1)^2,e-2 \rangle}
&\rt{\langle e-3,e-1 \rangle}
&\rt{\langle e-3 \rangle}
&\rt{\langle e-2,e-3 \rangle}
\\\hline
\langle e-1,e-1 \rangle&1&\cdot&\cdot&\cdot&\cdot\\
\langle (e-1)^2,e-2 \rangle&\cdot&1&\cdot&\cdot&\cdot\\
\langle e-3,e-1 \rangle&v&v&1&\cdot&\cdot\\
\langle e-3 \rangle&v^2&v^2&v&1&\cdot\\
\langle e-2,e-3 \rangle&v^3&v&v^2&v&1
\end{array}
\]
\end{answer}
Once again, \cref{thm:runnerrem} allows us to reduce to the above case in characteristic $0$.
Applying \cref{thm:wt3adj}, we see that:
$\langle e-1\rangle$ induces semi-simply to $\langle e-1\rangle$;
$\langle e-2, e-1\rangle$ induces semi-simply to $\langle e-2\rangle$;
$\langle (e-1)^2, e-2\rangle$ induces semi-simply to $\langle e-2, e-1\rangle$, but no partition induces semi-simply to $\langle (e-2)^2, e-1\rangle$;
$\langle e-1, e-2\rangle$ induces semi-simply to $\langle e-1, e-2\rangle$;
$\langle e-1, e-1\rangle$ induces semi-simply to $\langle e-1, e-1\rangle$ but no partition induces semi-simply to $\langle (e-1)^3\rangle$;
and no partition induces semi-simply to $\langle e-2, e-2\rangle$.
Since our chosen partitions do not include $\langle e-1\rangle$, $\langle e-2, e-1\rangle$, or $\langle e-1, e-2\rangle$, $a_{\nu\mu} = \delta_{\nu\mu}$ for all $\mu$ among them, and we see that the above submatrix is characteristic-free.
By \cref{prop:matrixtrick}, these blocks are therefore strictly wild and brick-wild.

\item The weight 3 block with core $\rho^{(3)}=(4,2)$.
We use the partitions $\langle (e-1)^2,e-2 \rangle$, $\langle e-2, e-1 \rangle$, $\langle e-3, e-1 \rangle$, $\langle e-3 \rangle$, and $\langle e-2,e-3 \rangle$.
Compared to \cite[Section~5.3]{als23}, we have used $\langle e-2,e-3 \rangle = (4^2,3,2^2)$ instead of $\langle e-2 \rangle$, and include the additional partition $\langle (e-1)^2,e-2 \rangle = (7,5,2,1)$ in order to yield a submatrix of the graded decomposition matrix of the form~(\ref{targetmatrixaltsquare2}).
Here we also include $\langle (e-1)^2,e-2 \rangle = (7,5,2,1)$, yielding a submatrix of the graded decomposition matrix of the form~(\ref{targetmatrixaltsquare2}).
\begin{answer}
\[
\begin{array}{r|ccccc}
&\rt{7,5,2,1}
&\rt{7,4,3,1}
&\rt{7,3^2,1^2}
&\rt{5^2,3,1^2}
&\rt{\mathrlap{4^2,3,2^2}\hphantom{\langle (e-1)^2,e-2 \rangle}}
\\\hline
7,5,2,1&1&\cdot&\cdot&\cdot&\cdot\\
7,4,3,1&v&1&\cdot&\cdot&\cdot\\
7,3^2,1^2&v^2&v&1&\cdot&\cdot\\
5^2,3,1^2&v^3&v^2&v&1&\cdot\\
4^2,3,2^2&\cdot&v&v^2&v&1
\end{array}
\qquad = \qquad
\begin{array}{r|ccccc}
&\rt{\langle (e-1)^2,e-2 \rangle}
&\rt{\langle e-2,e-1 \rangle}
&\rt{\langle e-3,e-1 \rangle}
&\rt{\langle e-3 \rangle}
&\rt{\langle e-2,e-3 \rangle}
\\\hline
\langle (e-1)^2,e-2 \rangle&1&\cdot&\cdot&\cdot&\cdot\\
\langle e-2,e-1 \rangle&v&1&\cdot&\cdot&\cdot\\
\langle e-3,e-1 \rangle&v^2&v&1&\cdot&\cdot\\
\langle e-3 \rangle&v^3&v^2&v&1&\cdot\\
\langle e-2,e-3 \rangle&\cdot&v&v^2&v&1
\end{array}
\]
\end{answer}
Once again, \cref{thm:runnerrem} allows us to reduce to the above case in characteristic $0$.
By \cref{thm:wt3adj}, we see that:
$\langle e-1\rangle$ induces semi-simply to $\langle e-1\rangle$;
$\langle e-2\rangle$ induces semi-simply to $\langle e-2\rangle$;
$\langle e-2, e-1\rangle$ induces semi-simply to $\langle e-2, e-1\rangle$, but no partition induces semi-simply to $\langle (e-2)^2, e-1\rangle$;
$\langle e-1, e-2\rangle$ induces semi-simply to $\langle e-1, e-2\rangle$;
$\langle e-1, e-1\rangle$ induces semi-simply to $\langle e-1, e-1\rangle$;
and no partition induces semi-simply to $\langle e-2, e-2\rangle$.
Since our chosen partitions do not include $\langle e-1\rangle$, $\langle e-2\rangle$, $\langle e-1, e-2\rangle$, or $\langle e-1, e-1\rangle$, $a_{\nu\mu} = \delta_{\nu\mu}$ for all $\mu$ among them, and we see that the above submatrix is characteristic-free.
By \cref{prop:matrixtrick}, these blocks are therefore strictly wild and brick-wild.
\end{enumerate}

\subsection{$p_{e-1} - p_{e-2} < e$ and $p_{e-2} - p_{e-3} > e$}\label{subsec:wt3fourthcase}

As in \cite[Section~5.4]{als23}, the $e=3$ situation reduces to three Scopes classes, with minimal representatives being the blocks with cores $\sigma^{(1)}=(1^2)$, $\sigma^{(2)}=(2,1^2)$, and $\sigma^{(3)}=(2^2,1^2)$, respectively.
We handle these in turn below.

\begin{enumerate}
\item
The weight 3 block with core $\sigma^{(1)}=(1^2)$.
In \cite[Section~5.4]{als23}, we used the partitions $\langle e-1,e-1 \rangle$, $\langle e-2, e-2 \rangle$, $\langle (e-1)^2, e-2 \rangle$, and $\langle (e-2)^2,e-1 \rangle$.
Here we also include $\langle e-2 \rangle = (9,2)$, yielding a submatrix of the graded decomposition matrix of the form~(\ref{targetmatrix1.5}).
\begin{answer}
\[
\begin{array}{r|ccccc}
&\rt{9,2}
&\rt{7,2^2}
&\rt{6,2^2,1}
&\rt{4^2,3}
&\rt{\mathrlap{4^2,2,1}\hphantom{\langle (e-1)^2,e-2 \rangle}}
\\\hline
9,2&1&\cdot&\cdot&\cdot&\cdot\\
7,2^2&v^2&1&\cdot&\cdot&\cdot\\
6,2^2,1&v&v&1&\cdot&\cdot\\
4^2,3&\cdot&v&\cdot&1&\cdot\\
4^2,2,1&v^2&v^2&v&v&1
\end{array}
\qquad = \qquad
\begin{array}{r|ccccc}
&\rt{\langle e-2 \rangle}
&\rt{\langle e-1,e-1 \rangle}
&\rt{\langle e-2,e-2 \rangle}
&\rt{\langle (e-1)^2,e-2 \rangle}
&\rt{\langle (e-2)^2,e-1 \rangle}
\\\hline
\langle e-2 \rangle&1&\cdot&\cdot&\cdot&\cdot\\
\langle e-1,e-1 \rangle&v^2&1&\cdot&\cdot&\cdot\\
\langle e-2,e-2 \rangle&v&v&1&\cdot&\cdot\\
\langle (e-1)^2,e-2 \rangle&\cdot&v&\cdot&1&\cdot\\
\langle (e-2)^2,e-1 \rangle&v^2&v^2&v&v&1
\end{array}
\]
\end{answer}
\cref{thm:runnerrem} allows us, as usual, to reduce to the above case in characteristic $0$.
Applying \cref{thm:wt3adj}, we see that:
$\langle e-1\rangle$ induces semi-simply to $\langle e-1\rangle$;
$\langle e-2, e-1\rangle$ induces semi-simply to $\langle e-2\rangle$;
$\langle (e-1)^2, e-2\rangle$ induces semi-simply to $\langle e-2, e-1\rangle$, but no partition induces semi-simply to $\langle (e-2)^2, e-1\rangle$;
$\langle e-1, e-1\rangle$ induces semi-simply to $\langle e-1, e-2\rangle$ but no partition induces semi-simply to $\langle (e-1)^2, e-2\rangle$;
$\langle (e-2)^2, e-1\rangle$ induces semi-simply to $\langle e-2, e-2\rangle$ but no partition induces semi-simply to $\langle (e-2)^3\rangle$;
and no partition induces semi-simply to $\langle e-1, e-1\rangle$.
Since our chosen partitions do not include $\langle e-1\rangle$ or $\langle e-2, e-1\rangle$, $a_{\nu\mu} = \delta_{\nu\mu}$ for all $\mu$ among them, and we see that the above submatrix is characteristic-free.
By \cref{prop:matrixtrick}, these blocks are therefore strictly wild and brick-wild.

\item
The weight 3 block with core $\sigma^{(2)}=(2,1^2)$.
In \cite[Section~5.4]{als23}, we used the partitions $\langle e-2 \rangle$, $\langle e-1, e-2 \rangle$, $\langle e-1, e-1 \rangle$, and $\langle e-2, e-2 \rangle$.
Here we also include $\langle (e-2)^2,e-1 \rangle = (5,4,2^2)$, yielding a submatrix of the graded decomposition matrix of the form~(\ref{targetmatrixalt}).
\begin{answer}
\[
\begin{array}{r|ccccc}
&\rt{9,3,1}
&\rt{8,4,1}
&\rt{8,3,2}
&\rt{6,3,2^2}
&\rt{\mathrlap{5,4,2^2}\hphantom{\langle (e-2)^2,e-1 \rangle}}
\\\hline
9,3,1&1&\cdot&\cdot&\cdot&\cdot\\
8,4,1&v&1&\cdot&\cdot&\cdot\\
8,3,2&v^2&v&1&\cdot&\cdot\\
6,3,2^2&v&v^2&v&1&\cdot\\
5,4,2^2&v^2&v^3&v^2&v&1
\end{array}
\qquad = \qquad
\begin{array}{r|ccccc}
&\rt{\langle e-2 \rangle}
&\rt{\langle e-1,e-2 \rangle}
&\rt{\langle e-1,e-1 \rangle}
&\rt{\langle e-2,e-2 \rangle}
&\rt{\langle (e-2)^2,e-1 \rangle}
\\\hline
\langle e-2 \rangle&1&\cdot&\cdot&\cdot&\cdot\\
\langle e-1,e-2 \rangle&v&1&\cdot&\cdot&\cdot\\
\langle e-1,e-1 \rangle&v^2&v&1&\cdot&\cdot\\
\langle e-2,e-2 \rangle&v&v^2&v&1&\cdot\\
\langle (e-2)^2,e-1 \rangle&v^2&v^3&v^2&v&1
\end{array}
\]
\end{answer}
The result again follows by \cref{thm:runnerrem,thm:wt3adj}, as we see that:
$\langle e-1\rangle$ induces semi-simply to $\langle e-1\rangle$;
$\langle e-2, e-1\rangle$ induces semi-simply to $\langle e-2\rangle$;
$\langle (e-1)^2, e-2\rangle$ induces semi-simply to $\langle e-2, e-1\rangle$;
$\langle e-1, e-1\rangle$ induces semi-simply to $\langle e-1, e-2\rangle$ but no partition induces semi-simply to $\langle (e-1)^2, e-2\rangle$;
$\langle (e-2)^2, e-1\rangle$ induces semi-simply to $\langle e-2, e-2\rangle$ but no partition induces semi-simply to $\langle (e-2)^3\rangle$;
and no partition induces semi-simply to $\langle e-1, e-1\rangle$.
Since our chosen partitions do not include $\langle e-1\rangle$, $\langle e-2, e-1\rangle$,  or $\langle (e-1)^2, e-2\rangle$, $a_{\nu\mu} = \delta_{\nu\mu}$ for all $\mu$ among them, and we see that the above submatrix is characteristic-free.
By \cref{prop:matrixtrick}, these blocks are therefore strictly wild and brick-wild.

\item
The weight 3 block with core $\sigma^{(3)}=(2^2,1^2)$.
In \cite[Section~5.4]{als23}, we used the partitions $\langle e-1 \rangle$, $\langle e-2 \rangle$, $\langle e-1, e-2 \rangle$, and $\langle e-1, e-1 \rangle$, but this resulted in the need for separate treatment for $p=2$.
Instead, we will use the partitions $\langle e-2 \rangle$, $\langle e-1,e-2 \rangle$, $\langle e-1, e-1 \rangle$, $\langle e-2, e-2 \rangle$, and $\langle (e-1)^3 \rangle$, yielding a submatrix of the graded decomposition matrix of the form~(\ref{targetmatrixaltsquare1.5}).
\begin{answer}
\[
\begin{array}{r|ccccc}
&\rt{10,3,1^2}
&\rt{8,5,1^2}
&\rt{8,3^2,1}
&\rt{7,3^2,2}
&\rt{\mathrlap{5,3^2,2^2}\hphantom{\langle e-2,e-2 \rangle}}
\\\hline
10,3,1^2&1&\cdot&\cdot&\cdot&\cdot\\
8,5,1^2&v&1&\cdot&\cdot&\cdot\\
8,3^2,1&v^2&v&1&\cdot&\cdot\\
7,3^2,2&v&v^2&v&1&\cdot\\
5,3^2,2^2&v^3&\dot&v&v^2&1
\end{array}
\qquad = \qquad
\begin{array}{r|ccccc}
&\rt{\langle e-2 \rangle}
&\rt{\langle e-1,e-2 \rangle}
&\rt{\langle e-1,e-1 \rangle}
&\rt{\langle e-2,e-2 \rangle}
&\rt{\langle (e-1)^3 \rangle}
\\\hline
\langle e-2 \rangle&1&\cdot&\cdot&\cdot&\cdot\\
\langle e-1,e-2 \rangle&v&1&\cdot&\cdot&\cdot\\
\langle e-1,e-1 \rangle&v^2&v&1&\cdot&\cdot\\
\langle e-2,e-2 \rangle&v&v^2&v&1&\cdot\\
\langle (e-1)^3 \rangle&v^3&\cdot&v&v^2&1
\end{array}
\]
\end{answer}
We previously missed that such uniform treatment was possible, so give some more detail here, following \cref{thm:runnerrem,thm:wt3adj} as usual.
We see that:
$\langle e-1\rangle$ induces semi-simply to $\langle e-1\rangle$;
$\langle e-2, e-1\rangle$ induces semi-simply to $\langle e-2\rangle$;
$\langle (e-1)^2, e-2\rangle$ induces semi-simply to $\langle e-2, e-1\rangle$;
$\langle e-1, e-1\rangle$ induces semi-simply to $\langle e-1, e-2\rangle$ but no partition induces semi-simply to $\langle (e-1)^2, e-2\rangle$;
$\langle (e-2)^2, e-1\rangle$ induces semi-simply to $\langle e-2, e-2\rangle$;
and no partition induces semi-simply to $\langle e-1, e-1\rangle$.
Since our chosen partitions do not include $\langle e-1\rangle$, $\langle e-2, e-1\rangle$, $\langle (e-2)^2, e-1\rangle$, or $\langle (e-1)^2, e-2\rangle$, $a_{\nu\mu} = \delta_{\nu\mu}$ for all $\mu$ among them, we see that the above submatrix is characteristic-free.
By \cref{prop:matrixtrick}, these blocks are therefore strictly wild and brick-wild.
\end{enumerate}

\subsection{$p_{e-1} - p_{e-2} > e$ and $p_{e-2} - p_{e-3} > e$}\label{subsec:wt3fifthcase}

As in \cite[Section~5.5]{als23}, we apply \cref{thm:runnerrem} to reduce to the $e=3$ case, and have four Scopes classes to consider, with minimal representatives being the blocks with cores $\kappa_{(e+1,e+1)} = (3,1^2)$, $\kappa_{(2e+1,e+1)} = (5,3,1^2)$, $\kappa_{(e+1,2e+1)} = (4,2^2,1^2)$, and $\kappa_{(2e+1,2e+1)} = (6,4,2^2,1^2)$, respectively.
We handle these in turn below.

\begin{enumerate}
\item The weight 3 block with core $\kappa_{(e+1,e+1)} = (3,1^2)$.
In \cite[Section~5.5]{als23}, we used the partitions $\langle e-1, e-1 \rangle = (9,4,1)$, $\langle e-1,e-2 \rangle = (9,3,2)$, $\langle (e-1)^2,e-2 \rangle = (6,4^2)$, and $\langle (e-1)^3 \rangle = (6,4,2^2)$ for $p\neq 2$.
Also assuming that $p\neq 2$, we may include $\langle e-2 \rangle = (8,4,2)$, yielding a submatrix of the graded decomposition matrix of the form~(\ref{targetmatrixalt2}).
\begin{answer}
\[
\begin{array}{r|ccccc}
&\rt{9,4,1}
&\rt{9,3,2}
&\rt{8,4,2}
&\rt{6,4^2}
&\rt{\mathrlap{6,4,2^2}\hphantom{\langle (e-1)^2,e-2 \rangle}}
\\\hline
9,4,1&1&\cdot&\cdot&\cdot&\cdot\\
9,3,2&v&1&\cdot&\cdot&\cdot\\
8,4,2&v^2&v&1&\cdot&\cdot\\
6,4^2&\cdot&v&v^2&1&\cdot\\
6,4,2^2&v&v^2&v^3&v&1
\end{array}
\qquad = \qquad
\begin{array}{r|ccccc}
&\rt{\langle e-1,e-1 \rangle}
&\rt{\langle e-1,e-2 \rangle}
&\rt{\langle e-2 \rangle}
&\rt{\langle (e-1)^2,e-2 \rangle}
&\rt{\langle (e-1)^3 \rangle}
\\\hline
\langle e-1,e-1 \rangle&1&\cdot&\cdot&\cdot&\cdot\\
\langle e-1,e-2 \rangle&v&1&\cdot&\cdot&\cdot\\
\langle e-2 \rangle&v^2&v&1&\cdot&\cdot\\
\langle (e-1)^2,e-2 \rangle&\cdot&v&v^2&1&\cdot\\
\langle (e-1)^3 \rangle&v&v^2&v^3&v&1
\end{array}
\]
\end{answer}
\cref{thm:runnerrem} allows us, as usual, to reduce to the above case in characteristic zero.
We apply \cref{thm:wt3adj} and see that:
$\langle e-1\rangle$ induces semi-simply to $\langle e-1\rangle$ and $\langle e-2\rangle$ induces almost semi-simply to $\langle (e-1)^3\rangle$;
$\langle e-2, e-1\rangle$ induces semi-simply to $\langle e-2\rangle$ and $\langle e-3\rangle$ induces almost semi-simply to $\langle (e-2)^3\rangle$;
$\langle (e-1)^2, e-2\rangle$ induces semi-simply to $\langle e-2, e-1\rangle$ and 
$\langle (e-2)^2, e-1\rangle$ induces semi-simply to $\langle (e-2)^2, e-1\rangle$;
$\langle e-1, e-2\rangle$ induces semi-simply to $\langle e-1, e-2\rangle$ and 
$\langle (e-1)^3\rangle$ induces semi-simply to $\langle (e-1)^2, e-2\rangle$;
$\langle e-1, e-1\rangle$ induces semi-simply to $\langle e-1, e-1\rangle$, but no partition induces semi-simply to $\langle (e-1)^3 \rangle$;
$\langle e-2, e-2\rangle$ induces semi-simply to $\langle e-2, e-2\rangle$, but no partition induces semi-simply to $\langle (e-2)^3\rangle$.
Since our chosen partitions do not include $\langle e-1\rangle$ or $\langle e-2, e-1\rangle$, $a_{\nu\mu} = \delta_{\nu\mu}$ for all $\mu$ among them if $p\neq 2$, and we see that the above submatrix is the same for any $p\neq 2$.
By \cref{prop:matrixtrick}, these blocks are therefore strictly wild and brick-wild.

It remains to check the case $p=2$, in which the above partitions do not give a submatrix that is the same in characteristics $0$ and $2$.
In \cite[Section~5.5]{als23}, we used the four partitions $\langle e-1, e-2 \rangle = (9,3,2)$, $\langle e-2 \rangle = (8,4,2)$, $\langle e-2, e-1 \rangle = (6^2,2)$, and $\langle (e-1)^2, e-2 \rangle = (6,4^2)$ for $p=2$.
Here we include $\langle e-1, e-1 \rangle = (9,4,1)$, and by the list above, we may check that our submatrix remains unchanged in characteristic $2$.

To do this, we note that we do not include the partition $\langle e-1\rangle$, but still have
\[
a_{\langle e-3 \rangle \langle e-2, e-1 \rangle} = a_{\langle (e-1)^3 \rangle \langle e-1, e-2 \rangle} = a_{\langle (e-2)^2, e-1 \rangle \langle (e-1)^2, e-2 \rangle} = 1
\]
Since $\langle e-3 \rangle$, $\langle (e-1)^3 \rangle$, and $\langle (e-2)^2, e-1 \rangle$ are all dominated by our five partitions, the above adjustment matrix entries also do not change any $d^{e,p}_{\la \mu}(v)$ in our submatrix, so these blocks are again strictly wild and brick-wild.
In this case, the submatrix is of the form~(\ref{targetmatrixaltsquare2}).
\begin{answer}
\[
\begin{array}{r|ccccc}
&\rt{9,4,1}
&\rt{9,3,2}
&\rt{8,4,2}
&\rt{6^2,2}
&\rt{\mathrlap{6,4^2}\hphantom{\langle (e-1)^2,e-2 \rangle}}
\\\hline
9,4,1&1&\cdot&\cdot&\cdot&\cdot\\
9,3,2&v&1&\cdot&\cdot&\cdot\\
8,4,2&v^2&v&1&\cdot&\cdot\\
6^2,2&\cdot&\cdot&v&1&\cdot\\
6,4^2&\cdot&v&v^2&v&1
\end{array}
\qquad = \qquad
\begin{array}{r|ccccc}
&\rt{\langle e-1,e-1 \rangle}
&\rt{\langle e-1,e-2 \rangle}
&\rt{\langle e-2 \rangle}
&\rt{\langle e-2,e-1 \rangle}
&\rt{\langle (e-1)^2,e-2 \rangle}
\\\hline
\langle e-1,e-1 \rangle&1&\cdot&\cdot&\cdot&\cdot\\
\langle e-1,e-2 \rangle&v&1&\cdot&\cdot&\cdot\\
\langle e-2 \rangle&v^2&v&1&\cdot&\cdot\\
\langle e-2,e-1 \rangle&\cdot&\cdot&v&1&\cdot\\
\langle (e-1)^2,e-2 \rangle&\cdot&v&v^2&v&1
\end{array}
\]
\end{answer}

\item The weight 3 block with core $\kappa_{(e+1,2e+1)} = (4,2^2,1^2)$.
In \cite[Section~5.5]{als23}, when $p\neq 2$ we used the partitions $\langle e-1,e-2 \rangle = (10,4,3,1^2)$, $\langle (e-1)^2,e-2 \rangle = (7,5^2,1^2)$, $\langle (e-1)^3 \rangle = (7,5,3^2,1)$, and $\langle (e-2)^2,e-1 \rangle = (7,4,3^2,2)$.
Assuming $p\neq 2$, we also include the partition $\langle e-1,e-1 \rangle = (10,5,2,1^2)$, yielding a submatrix of the graded decomposition matrix of the form~(\ref{targetmatrixalt}), and also~(\ref{targetmatrixaltsquare2}).
\begin{answer}
\[
\begin{array}{r|ccccc}
&\rt{10,5,2,1^2}
&\rt{10,4,3,1^2}
&\rt{7,5^2,1^2}
&\rt{7,5,3^2,1}
&\rt{\mathrlap{7,4,3^2,2}\hphantom{\langle (e-1)^2,e-2 \rangle}}
\\\hline
10,5,2,1^2&1&\cdot&\cdot&\cdot&\cdot\\
10,4,3,1^2&v&1&\cdot&\cdot&\cdot\\
7,5^2,1^2&\cdot&v&1&\cdot&\cdot\\
7,5,3^2,1&v&v^2&v&1&\cdot\\
7,4,3^2,2&v^2&v&\cdot&v&1
\end{array}
\qquad = \qquad
\begin{array}{r|ccccc}
&\rt{\langle e-1,e-1 \rangle}
&\rt{\langle e-1,e-2 \rangle}
&\rt{\langle (e-1)^2,e-2 \rangle}
&\rt{\langle (e-1)^3 \rangle}
&\rt{\langle (e-2)^2,e-1 \rangle}
\\\hline
\langle e-1,e-1 \rangle&1&\cdot&\cdot&\cdot&\cdot\\
\langle e-1,e-2 \rangle&v&1&\cdot&\cdot&\cdot\\
\langle (e-1)^2,e-2 \rangle&\cdot&v&1&\cdot&\cdot\\
\langle (e-1)^3 \rangle&v&v^2&v&1&\cdot\\
\langle (e-2)^2,e-1 \rangle&v^2&v&\cdot&v&1
\end{array}
\]
\end{answer}
We apply \cref{thm:wt3adj} and see that:
$\langle e-1\rangle$ induces semi-simply to $\langle e-1\rangle$;
$\langle e-2, e-1\rangle$ induces semi-simply to $\langle e-2\rangle$ and $\langle (e-2)^3\rangle$ induces semi-simply to $\langle (e-2)^3\rangle$; 
$\langle (e-1)^2, e-2\rangle$ induces semi-simply to $\langle e-2, e-1\rangle$ and 
$\langle (e-2)^2, e-1\rangle$ induces semi-simply to $\langle (e-2)^2, e-1\rangle$; 
$\langle e-1, e-2\rangle$ induces semi-simply to $\langle e-1, e-2\rangle$ and 
$\langle (e-1)^3\rangle$ induces semi-simply to $\langle (e-1)^2, e-2\rangle$; 
$\langle e-1, e-1\rangle$ induces semi-simply to $\langle e-1, e-1\rangle$, but no partition induces semi-simply to $\langle (e-1)^3 \rangle$; 
$\langle e-2, e-2\rangle$ induces semi-simply to $\langle e-2, e-2\rangle$.
Since our chosen partitions do not include $\langle e-1\rangle$, $\langle e-2, e-1\rangle$, or $\langle e-2, e-2 \rangle$, $a_{\nu\mu} = \delta_{\nu\mu}$ for all $\mu$ among them if $p\neq 2$, and we see that the above submatrix is the same for any $p\neq 2$.
By \cref{prop:matrixtrick}, these blocks are therefore strictly wild and brick-wild.

It remains to check the case $p=2$, in which the above partitions do not give a submatrix that is the same in characteristics $0$ and $2$.
In \cite[Section~5.5]{als23}, we used the four partitions $\langle e-1, e-2 \rangle = (10,4,3,1^2)$, $\langle e-2 \rangle = (9,5,3,1^2)$, $\langle e-2, e-1 \rangle = (7^2,3,1^2)$, and $\langle (e-1)^2, e-2 \rangle = (7,5^2,1^2)$ for $p=2$.
Here we include $\langle e-1, e-1 \rangle = (10,5,2,1^2)$, and by the list above, we may check that our submatrix remains unchanged in characteristic $2$.

To do this, we note that we do not include the partition $\langle e-1\rangle$, but still have
\[
a_{\langle (e-2)^3 \rangle \langle e-2,e-1 \rangle} = a_{\langle (e-1)^3 \rangle \langle e-1, e-2 \rangle} = a_{\langle (e-2)^2, e-1 \rangle \langle (e-1)^2, e-2 \rangle} = 1
\]
Since $\langle (e-2)^3 \rangle$, $\langle (e-1)^3 \rangle$, and $\langle (e-2)^2, e-1 \rangle$ are all dominated by our five partitions, the $d^{e,p}_{\la \mu}(v)$ entries in our submatrix do not change, so these blocks are again strictly wild and brick-wild.
In this case, the submatrix is of the form~(\ref{targetmatrixaltsquare2}).
\begin{answer}
\[
\begin{array}{r|ccccc}
&\rt{10,5,2,1^2}
&\rt{10,4,3,1^2}
&\rt{9,5,3,1^2}
&\rt{7^2,3,1^2}
&\rt{\mathrlap{7,5^2,1^2}\hphantom{\langle (e-1)^2,e-2 \rangle}}
\\\hline
10,5,2,1^2&1&\cdot&\cdot&\cdot&\cdot\\
10,4,3,1^2&v&1&\cdot&\cdot&\cdot\\
9,5,3,1^2&v^2&v&1&\cdot&\cdot\\
7^2,3,1^2&\cdot&\cdot&v&1&\cdot\\
7,5^2,1^2&\cdot&v&v^2&v&1
\end{array}
\qquad = \qquad
\begin{array}{r|ccccc}
&\rt{\langle e-1,e-1 \rangle}
&\rt{\langle e-1,e-2 \rangle}
&\rt{\langle e-2 \rangle}
&\rt{\langle e-2,e-1 \rangle}
&\rt{\langle (e-1)^2,e-2 \rangle}
\\\hline
\langle e-1,e-1 \rangle&1&\cdot&\cdot&\cdot&\cdot\\
\langle e-1,e-2 \rangle&v&1&\cdot&\cdot&\cdot\\
\langle e-2 \rangle&v^2&v&1&\cdot&\cdot\\
\langle e-2,e-1 \rangle&\cdot&\cdot&v&1&\cdot\\
\langle (e-1)^2,e-2 \rangle&\cdot&v&v^2&v&1
\end{array}
\]
\end{answer}

\item The weight 3 block with core $\kappa_{(2e+1,e+1)} = (5,3,1^2)$.
As in \cite[Section~5.5]{als23} (when this block was discussed for $p=2$), it suffices to note that the $\#$-automorphism on $\hhh$ (e.g.~see \cite[Exercise~3.14]{mathas}) induces an isomorphism between this block and the one with core $(5,3,1^2)' = (4,2^2,1^2)$.

\item The weight 3 block with core $\kappa_{(2e+1,2e+1)} = (6,4,2^2,1^2)$.
In \cite[Section~5.5]{als23}, we used the partitions $\langle e-1,e-2 \rangle = (12,4^2,3,1^2)$, $\langle (e-1)^2,e-2 \rangle = (9,7,4,3,1^2)$, $\langle e-2,e-1 \rangle = (9,6,5,3,1^2)$, and $\langle (e-2)^2,e-1 \rangle = (9,4^2,3^2,2)$.
Assuming $p\neq 2$, we also include the partition $\langle (e-1)^3 \rangle = (9,7,5,2,1^2)$, yielding a submatrix of the graded decomposition matrix of the form~(\ref{targetmatrix2}).
\begin{answer}
\[
\begin{array}{r|ccccc}
&\rt{12,4^2,3,1^2}
&\rt{9,7,5,2,1^2}
&\rt{9,7,4,3,1^2}
&\rt{9,6,5,3,1^2}
&\rt{\mathrlap{9,4^2,3^2,2}\hphantom{\langle (e-1)^2,e-2 \rangle}}
\\\hline
12,4^2,3,1^2&1&\cdot&\cdot&\cdot&\cdot\\
9,7,5,2,1^2&\cdot&1&\cdot&\cdot&\cdot\\
9,7,4,3,1^2&\cdot&v&1&\cdot&\cdot\\
9,6,5,3,1^2&v&v^2&v&1&\cdot\\
9,4^2,3^2,2&v&\cdot&v&\cdot&1
\end{array}
\qquad = \qquad
\begin{array}{r|ccccc}
&\rt{\langle e-1,e-2 \rangle}
&\rt{\langle (e-1)^3 \rangle}
&\rt{\langle (e-1)^2,e-2 \rangle}
&\rt{\langle e-2,e-1 \rangle}
&\rt{\langle (e-2)^2,e-1 \rangle}
\\\hline
\langle e-1,e-2 \rangle&1&\cdot&\cdot&\cdot&\cdot\\
\langle (e-1)^3 \rangle&\cdot&1&\cdot&\cdot&\cdot\\
\langle (e-1)^2,e-2 \rangle&\cdot&v&1&\cdot&\cdot\\
\langle e-2,e-1 \rangle&v&v^2&v&1&\cdot\\
\langle (e-2)^2,e-1 \rangle&v&\cdot&v&\cdot&1
\end{array}
\]
\end{answer}
If $p=2$ the block with core $(6,4,2^2,1^2)$ (i.e.~the Rouquier block) is not susceptible to our methods when $e=3$, and we must treat it separately.

For now we will assume that $e\geq 4$.
As in \cite[Section~5.5]{als23}, we may reduce the problem to ten Scopes classes when $e=4$, and note that eight of them are amenable to a further runner removal to land in one of the previously handled cases.
The remaining two blocks, in the notation of loc.~cit., have cores $\kappa_{2e+1, 2e+1, e+1} = (11,8,5^2,3^2,1^3)$ and $\kappa_{2e+1, 2e+1, 2e+1} = (12,9,6^2,4^2,2^3,1^3)$.
In both cases, we used the partitions $\langle (e-1)^2, e-2 \rangle$, $\langle (e-1)^2, e-3 \rangle$, $\langle (e-2)^2, e-1 \rangle$, and $\langle e-3, e-2, e-1 \rangle$ in \cite[Section~5.5]{als23} to show that these blocks are brick-infinite.
In both cases, we include $\langle e-1, e-1\rangle$ to yield a submatrix of the form~(\ref{targetmatrix}) which is identical in characteristics 0 and 2, since we avoid all partitions of the form $\langle i \rangle$ and $\langle i,k \rangle$ (which each induce semi-simply to $\langle i \rangle$ and $\langle i, k \rangle$, respectively).
\begin{answer}
{\scalefont{.96}
\begin{align*}
\begin{array}{r|ccccc}
&\rt{19,12,5^2,3^2,1^3}
&\rt{15,12,8,6,3^2,1^3}
&\rt{15,12,5^2,3^3,2^2}
&\rt{15,8^2,6^2,4,1^3}
&\rt{15,8^2,6,3^3,2^2}
\\\hline
19,12,5^2,3^2,1^3&1&\cdot&\cdot&\cdot&\cdot\\
15,12,8,6,3^2,1^3&v&1&\cdot&\cdot&\cdot\\
15,12,5^2,3^3,2^2&\cdot&v&1&\cdot&\cdot\\
15,8^2,6^2,4,1^3&v^2&v&\cdot&1&\cdot\\
15,8^2,6,3^3,2^2&\cdot&v^2&v&v&1
\end{array}
\qquad &= \qquad
\begin{array}{r|ccccc}
&\rt{\langle e-1,e-1 \rangle}
&\rt{\langle (e-1)^2,e-2 \rangle}
&\rt{\langle (e-1)^2,e-3 \rangle}
&\rt{\langle (e-2)^2,(e-1) \rangle}
&\rt{\mathrlap{\langle e-3,e-2,e-1 \rangle}\hphantom{15,12,8,6,3^2,1^3}}
\\\hline
\langle e-1,e-1 \rangle&1&\cdot&\cdot&\cdot&\cdot\\
\langle (e-1)^2,e-2 \rangle&v&1&\cdot&\cdot&\cdot\\
\langle (e-1)^2,e-3 \rangle&\cdot&v&1&\cdot&\cdot\\
\langle (e-2)^2,(e-1) \rangle&v^2&v&\cdot&1&\cdot\\
\langle e-3,e-2,e-1 \rangle&\cdot&v^2&v&v&1
\end{array}\\
\\
\qquad &= \qquad
\begin{array}{r|ccccc}
&\rt{20,13,6^2,4^2,2^3,1^3}
&\rt{16,13,9,7,4^2,2^3,1^3}
&\rt{16,13,6^2,4^3,3^2,1^3}
&\rt{16,9^2,7^2,5,2^3,1^3}
&\rt{16,9^2,7,4^3,3^2,1^3}
\\\hline
20,13,6^2,4^2,2^3,1^3&1&\cdot&\cdot&\cdot&\cdot\\
16,13,9,7,4^2,2^3,1^3&v&1&\cdot&\cdot&\cdot\\
16,13,6^2,4^3,3^2,1^3&\cdot&v&1&\cdot&\cdot\\
16,9^2,7^2,5,2^3,1^3&v^2&v&\cdot&1&\cdot\\
16,9^2,7,4^3,3^2,1^3&\cdot&v^2&v&v&1
\end{array}
\end{align*}
}
\end{answer}
\end{enumerate}

Finally, we return to the Rouquier block when $e=3$ and $p=2$.
In \cite[Section~5.5]{als23}, we used the decomposition matrix along with some homomorphisms between Specht modules to determine that the Gabriel quiver contained a copy of $A_3^{(1)}$ with vertices given by $\la_1 := (15,4,2^2,1^2)$, $\la_3 := (12,4^2,3,1^2)$, $\la_4 := (9,7,5,2,1^2)$, and $\la_5 := (9,7,4,3,1^2)$.
But in the process, we also used the simple Specht module indexed by $\la_2 := (12,7,2^2,1^2)$, showing that $\Rad(\spe{\la_3}) \cong \D{\la_1} \oplus \D{\la_2}$.
In particular, $\D{\la_2}$ extends $\D{\la_3}$, so that we actually have a copy of $A_3^{(1)\wedge}$ in the Gabriel quiver, and so by \cref{cor:subquiver}, these blocks are strictly wild and brick-wild.

\section{Blocks of weight at least 4}\label{sec:highwt}

Since we will deal with blocks of arbitrarily large weight, we change our conventions for partitions, switching to use $e$-quotient notation, ordered by the fixed integers $p_i$, as in \cite[Section~6]{als23}.
The integers $p_i$ are determined by our fixed core, $\rho$, and we will consider the weight $w$ block $B(\rho,w)$ whose $e$-core is $\rho$.

Ordering the components in the $e$-quotient starting with the runner containing the $p_{e-1}$ position, and then the runner containing the $p_{e-2}$ position, and so on until the $p_0$ runner, we may read a partition from each runner, considered as a $1$-runner abacus display.
This uniquely determines an $e$-quotient for any partition in $B(\rho,w)$ (i.e.~it is not only defined up to some cyclic permutation, since we are fixing our order, and is independent of the chosen $r$).
We will use the shorthand notation $\vn^k$ to denote a string of $k$ components, each equal to the empty partition, in the $e$-quotient notation for a partition $\la$.
See \cite[Section~6]{als23} for illustrative examples of $e$-quotients.

\begin{lem}\label{lem:wt2rouquish}
Suppose $e\geq 4$ and $\rho$ is a core satisfying $p_{e-1} - p_{e-2} >e$.
Define five partitions as follows.
\[
\la^{(0)} = ((1^2),\vn^{e-1}), \qquad \la^{(1)} = ((1),(1),\vn^{e-2}), \qquad \la^{(2)} = ((1),\vn,(1),\vn^{e-3}).
\]
\begin{enumerate}[label=(\roman*)]
\item
If $p_{e-2} - p_{e-3} <e$, then define 
\[
\la^{(3)} = (\vn,(2),\vn^{e-2}), \qquad \la^{(4)} = (\vn,\vn,(2),\vn^{e-3}).
\]

\item
If $p_{e-2} - p_{e-3} >e$, then define 
\[
\la^{(3)} = (\vn,(1^2),\vn^{e-2}), \qquad \la^{(4)} = (\vn,(1),(1),\vn^{e-3}).
\]
\end{enumerate}
Then the five partitions $\la^{(0)}$, $\la^{(1)}$, $\la^{(2)}$, $\la^{(3)}$, and $\la^{(4)}$ give a submatrix of the graded decomposition matrix which is identical in any characteristic, and is of the form~(\ref{targetmatrix}).
It follows from \cref{prop:matrixtrick} that $B(\rho,2)$ is strictly wild.
\end{lem}

\begin{proof}
The result follows almost immediately from \cite[Lemma~6.1]{als23}.
There, we used runner removal to reduce the computation down to $e=4$, and may do the same here.
We have added the partition $\la^{(0)} = ((1^2),\vn^{e-1})$, and so must check that in each Scopes class -- which have cores $(3)$, $(4,1^2)$, and $(4,1^3)$ satisfying $p_{e-2} - p_{e-3} <e$, and cores $(5,2^2)$ and $(6,3^2,1^3)$ satisfying $p_{e-2} - p_{e-3} >e$ -- we get the following matrix, of the form~(\ref{targetmatrix}), in any characteristic.
\begin{equation*}
\begin{pmatrix}
1&\cdot&\cdot&\cdot&\cdot\\
v&1&\cdot&\cdot&\cdot\\
\cdot&v&1&\cdot&\cdot\\
\cdot&v&\cdot&1&\cdot\\
\cdot&v^2&v&v&1
\end{pmatrix}
\end{equation*}

We may easily verify the result in characteristic $0$, for example by the LLT algorithm.
Since $B(\rho,2)$ is a weight 2 block, the result now follows for $p\geq 3$, as in \cref{sec:wt2}.
By applying \cref{thm:wt2adjust}, we may check that these submatrices are in fact identical when $p=2$, since $a_{\nu \mu}(v) = \delta_{\nu\mu}$ is only nonzero for $\mu = ((2),\vn^{e-1})$ or $(\vn,(2),\vn^{e-2})$, neither of which we have used among our chosen partitions.
\end{proof}

When $e=3$, we will also need to consider Scopes classes of weight $4$ blocks explicitly, and will encode a core $\rho$ by a triple $[s_0,s_1,s_2]$ encoding the number of beads on each runner of a truncated abacus display for $\rho$, as in \cite[Section~6]{als23}.

\begin{thm}\label{thm:mostblocks}
Let $e\geq 3$, and $w\geq 4$, and let $\rho$ be an $e$-core.
Unless $e=3$, $p_{e-1} - p_{e-2}>2e$, and  $p_{e-2} - p_{e-3} > e$, the weight $w$ block $B(\rho,w)$ of $\hhh$ is strictly wild.
\end{thm}

\begin{proof}
As in \cite[Theorem~6.2]{als23}, we tabulate the different possible cases and the necessary partitions to use in each.
Our table is almost identical to \cite[Table~1]{als23}, except that 1) the sixth row of that table doesn't work in our setting, so is omitted; 2) the final row of that table must be split into two separate cases for us, depending on whether $p_{e-1}-p_{e-3}<2e$ or $p_{e-1}-p_{e-3}>2e$; and 3) there was a typo in the case where $p_{e-1} - p_{e-3}<e$, $p_{e-1}-p_{e-4}>e$, where the partition $((w-2,1),\vn^{e-1})$ should have been $((w-2,2),\vn^{e-1})$, and we have corrected it accordingly.
The extra partition which we add in each case compared to \cite[Table~1]{als23} is included in the final column, although how it compares to the others in the dominance order varies from case to case.
Our results are given in Table \ref{tab1}, where the following shorthand is used: 
\begin{align*}
&\alpha \text{ denotes } p_i -p_j <e, && \beta \text{ denotes } e<p_i -p_j <2e, \\
 &\gamma \text{ denotes } 2e < p_i-p_j,&& \delta \text{ denotes } e< p_i -p_j,\\
 & \varepsilon \text{ denotes } p_i -p_j < 2e.
\end{align*}

As in loc.~cit., all cases in the table involve sliding the lowest bead on the abacus display for $\rho$ down $w-2$ spaces on each partition, and row removal -- \cref{thm:rowrem} -- reduces the calculation to the weight 2 case, which then follows by earlier work in \cref{sec:wt2,lem:wt2rouquish}.

If $p_{e-1} - p_{e-2}>2e$ and $p_{e-2} - p_{e-3}<e$, then the third row of the table handles this case for $p\neq2$, while the fourth row handles it for $p=2$ provided that $e\geq 4$.
To resolve the missing case $e=3$, $p=2$, we note that each such block is Scopes equivalent either to the block $B([1,1,s_2],w)$ for $s_2\geq 3$ or to the block $B([1,s_1,1],w)$ for $s_1\geq 4$.
But these blocks are Morita equivalent, by conjugation, to $B([1,s_1,s_1 -1],w)$ and $B([1,s_2,s_2],w)$, respectively.
But these blocks all satisfy $p_{e-1}-p_{e-2} <e$ and $p_{e-2}-p_{e-3}>2e$, so are strictly wild by the final row of the table.\qedhere

\afterpage{%
    \clearpage
    \thispagestyle{empty}
\begin{landscape}
\begin{center}
\renewcommand{\arraystretch}{1.2}
\[\begin{array}{|m{10pt}m{10pt}m{10pt}|m{15pt}m{15pt}|m{10pt}m{10pt}m{10pt}m{10pt}|m{15pt}m{15pt}|lll|ll|l} \hline
\multicolumn{3}{|c|}{p_{e-1}{-}p_{e-2}} & \multicolumn{2}{c|}{p_{e-2}{-}p_{e-3}} & 
\multicolumn{4}{c|}{p_{e-1}{-}p_{e-3}} & \multicolumn{2}{c|}{p_{e-1}{-}p_{e-4}} & \multicolumn{5}{c|}{} \\ 
$\alpha$ &$\beta$&$\gamma$ &\multicolumn{1}{c}{\alpha} &\multicolumn{1}{c|}{\delta} & $\alpha$ &  $\beta$& $\gamma$ & $\varepsilon$ & \multicolumn{1}{c}{\alpha} & \multicolumn{1}{c|}{\delta} & \multicolumn{5}{c|}{} \\ \hline 
&$\bullet$&&\multicolumn{1}{c}{\bullet}&&&$\bullet$&&&&&((w-2,2), \vn^{e-1}) & ((w-2),(2),\vn^{e-2}) & ((w-2,1),(1),\vn^{e-2}) & \text{(\ref{targetmatrix})} &\text{\cref{subsec:wt2firstcase}} \\
&&&&&&&&&&&((w-2),\vn,(2),\vn^{e-3}) & ((w-2,1),\vn,(1),\vn^{e-3}) &&& \\ \hline 
&$\bullet$&&\multicolumn{1}{c}{\bullet} &&&&$\bullet$&&&&((w-2,2), \vn^{e-1}) & ((w-2),(2),\vn^{e-2}) &((w-2),\vn,(2),\vn^{e-3})& \text{(\ref{targetmatrix2})} &\text{\cref{subsec:wt2secondcase}}\\
&&&&&&&&&&&((w-2,1),(1),\vn^{e-2}) & ((w-2,1^2),\vn^{e-1}) &&&\\ \hline 
&&$\bullet$&\multicolumn{1}{c}{\bullet}&&&&&&& &((w-2,2), \vn^{e-1}) & ((w-2,1),(1),\vn^{e-2}) &((w-2,1^2),\vn^{e-1})& \text{(\ref{targetmatrix3})} &\text{\cref{subsec:wt2thirdcase}} \hspace{12pt} p\neq2\\
&&&&&&&&&&&((w-2),(2),\vn^{e-2}) & ((w-2),\vn,(2),\vn^{e-3}) &  && \\ \hline 
&&$\bullet$&\multicolumn{1}{c}{\bullet}&&&&&&& &((w-2,1),(1), \vn^{e-2}) & ((w-2,1),\vn,(1),\vn^{e-3}) &((w-2,1^2),\vn^{e-1})&\text{(\ref{targetmatrix})} & \text{\cref{lem:wt2rouquish}} \hspace{28pt} p=2  \\
&&&&&&&&&&&((w-2),(2),\vn^{e-2}) & ((w-2),\vn,(2),\vn^{e-3}) & &&\hspace{82pt} e \geq 4 \\ \hline 
&$\bullet$ & & &$\bullet$ &&&&&&& ((w-2,2),\vn^{e-1}) & ((w-2),(2),\vn^{e-2}) &((w-2),(1^2),\vn^{e-2})& \text{(\ref{targetmatrixalt})} &\text{\cref{subsec:wt2fourthcase}}\\
&&&&&&&&&&&((w-2,1),(1),\vn^{e-2}) & ((w-2,1^2),\vn^{e-1}) &&& \\ \hline
&&$\bullet$ &&$\bullet$ &&&&&&& ((w-2,1),(1),\vn^{(e-2)}) & ((w-2,1),\vn,(1),\vn^{e-3}) &((w-2,1^2),\vn^{e-1})&  \text{(\ref{targetmatrix})} &\text{\cref{lem:wt2rouquish}} \hspace{27pt} e \geq 4\\
&&&&&&&&&&&((w-2),(1^2),\vn^{e-2}) & ((w-2),(1),(1),\vn^{e-3}) & &&\\ \hline
&&&&&&&&&\multicolumn{1}{c}{\bullet} && ((w-2),(2),\vn^{e-2}) & ((w-2),\vn,(2),\vn^{e-3}) &((w-2),(1),(1),\vn^{e-3})& \text{(\ref{targetmatrix})} &\text{\cref{subsec:wt2firstcase}} \\
&&&&&&&&&&&((w-2),\vn^2,(2),\vn^{e-4}) & ((w-2),(1),\vn,(1),\vn^{e-4}) &&& \\ \hline
&&&&& $\bullet$&&&&& \multicolumn{1}{c|}{\bullet} & ((w-2),(2),\vn^{e-2}) & ((w-2),\vn,(2),\vn^{e-3}) &((w-2),(1),(1),\vn^{e-3})&  \text{(\ref{targetmatrix})} &\text{\cref{subsec:wt2firstcase}}\\
&&&&&&&&&&&((w-2,2),\vn^{e-1}) & ((w-2,1),(1),\vn^{e-2}) &&& \\ \hline 
$\bullet$ &&&\multicolumn{1}{c}{\bullet} &&& $\bullet$ &&&&& ((w-2),(2),\vn^{e-2}) &((w-2,2),\vn^{e-1}) &((w-2,1),(1),\vn^{e-2})&  \text{(\ref{targetmatrix})} &\text{\cref{subsec:wt2firstcase}}\\
&&&&&&&&&&& ((w-2),\vn,(2),\vn^{e-3}) & ((w-2),(1),(1),\vn^{e-3}) &&& \\ \hline 
$\bullet$ &&& & $\bullet$ && &&$\bullet$& &&((w-2),(2),\vn^{e-2}) & ((w-2,2),\vn^{e-1}) &((w-2),\vn,(2),\vn^{e-3})&  \text{(\ref{targetmatrix2})} &\text{\cref{subsec:wt2secondcase}} \\ 
&&&&&&&&&&&((w-2,1),(1),\vn^{e-2}) & ((w-2),(1^2),\vn^{e-3}) &&&  \\ \hline
$\bullet$ &&& & $\bullet$ && &$\bullet$& &&&((w-2),(2),\vn^{e-2}) & ((w-2,2),\vn^{e-1}) &((w-2,1^2),\vn^{e-1})&  \text{(\ref{targetmatrixalt})} &\text{\cref{subsec:wt2fourthcase}} \\ 
&&&&&&&&&&&((w-2,1),(1),\vn^{e-2}) & ((w-2),(1^2),\vn^{e-3}) &&&  \\ \hline 
\end{array}\]
\captionof{table}{A case-by-case analysis of \cref{thm:mostblocks}} \label{tab1}
\end{center}
\end{landscape}
    \clearpage
}
\end{proof}

It now only remains to show that the blocks $B(\rho,w)$ for $e=3$, $w\geq4$, $p_{e-1} - p_{e-2} > 2e$, and  $p_{e-2} - p_{e-3} > e$ are strictly wild.
We will first deal with the weight 4 blocks, and note that these fall into the following six Scopes classes.

\begin{alignat*}{3}
&[ 1, 2, 4 ], \qquad &&[ 1, 2, 5 ], \qquad &&[ 1, 3, 5 ]\\
&[ 1, 3, 6 ], &&[ 1, 4, 6 ], &&[ 1, 4, 7 ].
\end{alignat*}

Since the Scopes classes $[1,2,4]$ and $[1,2,5]$ are conjugate to the classes $[1,3,4]$ and $[1,4,5]$, respectively, and these classes are covered by \cref{thm:mostblocks}, it immediately follows that the corresponding blocks are strictly wild, and we may focus on the remaining four blocks.

\begin{prop}\label{prop:Rouqblocks}
Let $B = B([1,4,7],4)$ be the Rouquier block of weight $4$, and define the following partitions in $B$ by their $3$-quotients.

If $p\neq 3$, then we set 
\begin{alignat*}{3}
\la^{(0)} &= ((2,1),(1),\vn), \qquad &\la^{(1)} &= ((1^{2}),(1^2),\vn),\qquad 
&\la^{(2)} &= ((1),(2,1),\vn),\\
\la^{(3)} &= ((1),(1^3),\vn), &
\la^{(4)} &= (\vn,(2,1^2),\vn).
\end{alignat*}
Then the partial decomposition matrix corresponding to these five partitions is identical in characteristic $0$ and characteristic $p$ and is of the form~(\ref{targetmatrix}), and hence $B$ is strictly wild.

If $p=3$, then we set
\begin{alignat*}{3}
\la^{(1)} &= ((2^2),\vn^2), \qquad &\la^{(2)} &= ((2,1^2),\vn^2),\qquad 
&\la^{(3)} &= ((2,1),(1),\vn),\\
\la^{(4)} &= ((2),(2),\vn), &
\la^{(5)} &= ((2),(1^2),\vn), &\la^{(6)}&=((1),(1^3),\vn).
\end{alignat*}
Then the partial decomposition matrix corresponding to these six partitions is identical in characteristic $0$ and characteristic $3$ and is equal to~(\ref{targetmatrixstar}), and hence $B$ is strictly wild.
\end{prop}

\begin{proof}
For $p\neq 3$, the proof proceeds much the same as that of \cite[Proposition~6.4]{als23}, noting that by \cite{fay07wt4}, the adjustment matrix for weight 4 blocks is trivial unless $p=2$ or $3$.
Since we already found the submatrix indexed by $\la^{(1)}, \la^{(2)}, \la^{(3)},\la^{(4)}$ in \cite{als23}, and verified that it is the same submatrix in characteristics $0$ and $2$, we need only check the entries $d^{e,p}_{\la^{(i)} \la^{(0)}}(v)$ for $i=1,2,3,4$.

One may check using the LLT algorithm that the partial decomposition matrix is
\begin{equation*}
\begin{pmatrix}
1&\cdot&\cdot&\cdot&\cdot\\
v&1&\cdot&\cdot&\cdot\\
v^2&v&1&\cdot&\cdot\\
\cdot&v&\cdot&1&\cdot\\
\cdot&v^2&v&v&1
\end{pmatrix}.
\end{equation*}

Then we have
\[
d^{e,p}_{\la^{(i)}\la^{(0)}}(v) = d^{e,0}_{\la^{(i)}\la^{(0)}}(v) + \sum_{\nu \triangleleft \la^{(0)}} d^{e,0}_{\la^{(i)}\nu}(v)a_{\nu\la^{(0)}}(v),\label{eqn:adj} 
\]
where $A =(a_{\la\mu}(v))$ is the \emph{graded} adjustment matrix.
If $\nu,\mu \in B([1,4,7],4)$ have respective $3$-quotients $(\nu^{(0)},\nu^{(1)},\nu^{(2)})$ and $(\mu^{(0)},\mu^{(1)},\mu^{(2)})$, then we have  $a_{\nu\mu}(v) = 0$ unless $|\nu^{(i)}|=|\mu^{(i)}|$ for $i=0,1,2$, by \cite[Proposition~4.4]{jlm}.
Then $d^{3,2}_{\la^{(i)}\la^{(0)}}(v) = d^{3,0}_{\la^{(i)}\la^{(0)}}(v) + d^{3,0}_{\la^{(i)}\nu}(v)a_{\nu\la^{(0)}}(v)$ for $\nu = ((1^3),(1),\vn)$, and the partial decomposition matrices agree if $a_{\nu\la^{(0)}}(v)=0$.
One may use the Jantzen sum formula to compute that $d^{3,2}_{\nu\la^{(0)}}(v)=0$, and therefore that $a_{\nu\la^{(0)}}(v)=0$.

Next, suppose that $p=3$.
Since $e=p=3$, $B$ is isomorphic to the weight $4$ Rouquier block for the symmetric group in characteristic $3$, and we may apply the stronger \cite[Corollary~4.5]{jlm} to compute the adjustment matrix precisely.
Then one may compute that the corresponding partial decomposition matrix here is indeed (\ref{targetmatrixstar}) in both characteristic $0$ and $3$.
In fact, the only nonzero adjustment matrix entries of the form $a_{\nu\la^{(i)}}(v)$ are $a_{\nu\la^{(1)}}(v) = a_{\nu' \la^{(3)}}(v) = 1$ for $\nu = ((1^4),\vn^2)$ and $\nu' = ((1^3),(1),\vn)$, but $d^{e,0}_{\la^{(i)}\nu}(v) = d^{e,0}_{\la^{(i)}\nu'}(v) = 0$ for all $i$, so that the submatrix is indeed the same in characteristics $0$ and $3$, and the proof is complete.
\end{proof}

As our strategy is similar to that in \cite{als23}, the following lemma will allow us to prove that the remaining weight $4$ blocks are strictly wild.

\begin{lemc}{als23}{Lemma~6.6}\label{lem:rest}
Let $\la,\mu \in B(\rho,w)$ be such that $\mu$ is $e$-regular and both $\la$ and $\mu$ have exactly $k$ removable $i$-nodes.
Let $\bar{\la}$ and $\bar{\mu}$ be the partitions obtained by removing $k$ $i$-nodes from $\la$ and $\mu$ respectively.
Suppose that $\bar{\la}$ and $\bar{\mu}$ both have exactly $k$ addable $i$-nodes.
Then $\bar{\mu}$ is $e$-regular and $d^{e,p}_{\la\mu}(1) \geq d^{e,p}_{\bar{\la}\bar{\mu}}(1)$.
\end{lemc}

\begin{prop}\label{prop:NearRouqblocks}
Suppose that $s \in \{[1,4,6], [1,3,6], [1,3,5]\}$.
We define the following partitions in $B(s,4)$ by their $3$-quotients.
If $p\neq 3$, then we set 
\begin{alignat*}{3}
\la^{(0)} &= ((2,1),(1),\vn), \qquad &\la^{(1)} &= ((1^{2}),(1^2),\vn),\qquad 
&\la^{(2)} &= ((1),(2,1),\vn),\\
\la^{(3)} &= ((1),(1^3),\vn), &
\la^{(4)} &= (\vn,(2,1^2),\vn).
\end{alignat*}
Then the partial decomposition matrix corresponding to these five partitions is identical in characteristic $0$ and characteristic $p$ and is of the form~(\ref{targetmatrix}), and hence $B$ is strictly wild.

If $p=3$, then we set
\begin{alignat*}{3}
\la^{(1)} &= ((2^2),\vn^2), \qquad &\la^{(2)} &= ((2,1^2),\vn^2),\qquad 
&\la^{(3)} &= ((2,1),(1),\vn),\\
\la^{(4)} &= ((2),(2),\vn), &
\la^{(5)} &= ((2),(1^2),\vn), &\la^{(6)}&=((1),(1^3),\vn).
\end{alignat*}
Then the partial decomposition matrix corresponding to these six partitions in both characteristic $0$ and characteristic $3$ is equal to (\ref{targetmatrixstar})
and hence $B(s,4)$ is strictly wild.
\end{prop}

\begin{proof}
As in \cite[Proposition~6.7]{als23}, we prove the stronger claim that the proposition holds for all $s \in \{[1,4,7], [1,7,4],$ $[1,4,6], [4,1,7], [1,3,6], [1,6,3], [1,3,5]\}$ by noting that the LLT algorithm gives us the matrices if $p=0$, and give lower bounds on the entries for $p=2$ and $3$.
The case $s=[1,4,7]$ is proved in \cref{prop:Rouqblocks}, and by applying \cref{lem:rest} we obtain upper bounds on the entries for $[1,7,4]$, coming from those known decomposition numbers for the Rouquier block $[1,4,7]$.
These agree with the lower bounds, yielding the result.
We may then apply the same trick again to get from $[1,7,4]$ to $[7,1,4]=[1,4,6]$.
By an identical argument, we can pass from $[1,4,7]$ to $[4,1,7]$, then to $[4,7,1]=[1,3,6]$, then to $[7,4,1]=[1,6,3]$, and then finally to $[6,1,3]=[1,3,5]$.
The key is that our partitions have been carefully chosen so that we may apply \cref{lem:rest} -- the reader is invited to check these details.
\end{proof}

\begin{rem}
By \cref{thm:mostblocks,prop:Rouqblocks,prop:NearRouqblocks}, all weight 4 blocks are strictly wild.
\end{rem}

We now complete the proof of \cref{thm:MAIN} by handling the remaining blocks of weight at least $5$.

\begin{thm}\label{thm:wlargemain}
Suppose $e=3$, $p_{e-1}-p_{e-2}>2e$, $p_{e-2}-p_{e-3}>e$, $w \geq 5$, and let $s$ be a Scopes class for $w$.
Then the block $B(s,w)$ of $\hhh$ is strictly wild.
\end{thm}

\begin{proof}
We may write $s=[1,s_1,s_2]$.
First, suppose the Scopes class is $[1,s_1,s_2]$, with $s_1 \leq s_2$.
Since $[1,s_1,s_2]' = [1,s_2-s_1+1,s_2]$, and $B(s,w)$ and $B(s',w)$ are Morita equivalent, it suffices to consider the cases where $s_2-s_1 \leq s_1-1$, so we assume further that this inequality holds.
Since $p_{e-1}-p_{e-2} > 2e$, we may also assume that $s_2-s_1 \geq 2$.

If $s_2-s_1 = 2$, then $s_1 \geq 3$.
We set
\[
\la^{(0)} = ((w-3,2,1),\vn^2), \qquad \la^{(1)} = ((w-3,2),(1),\vn),
\]
and if $p\neq 2$ further set
\[
\la^{(2)} = ((w-3,1^2),(1),\vn), \qquad
\la^{(3)} = ((w-3,1^3),\vn^2), \qquad
\la^{(4)} = ((w-3,1),(1^2),\vn),
\]
while if $p=2$ we instead set
\[
\la^{(2)} = ((w-3),(3),\vn), \qquad
\la^{(3)} = ((w-3,1),(2),\vn), \qquad
\la^{(4)} = ((w-3,1^2),(1),\vn).
\]
Then by removing the first row of each partition, we reduce the computation of decomposition numbers to the case $B([1,s_1,s_1+1],3)$, which is Scopes equivalent to the block $B([1,3,4],3)$.
The resulting matrix is of the form (\ref{targetmatrixalt}) if $p\neq2$, or (\ref{targetmatrixaltsquare2}) if $p=2$, as shown in \cref{subsec:wt3fifthcase}(ii), so the result follows by \cref{prop:matrixtrick}.

Next, we suppose that $s_2-s_1=3$, so that we also have $s_1 \geq 4$.
If $p\neq 3$, we set 
\begin{alignat*}{3}
\la^{(0)} &= ((w-4,2,1),(1),\vn), \qquad &\la^{(1)} &= ((w-4,1^{2}),(1^2),\vn),\qquad 
&\la^{(2)} &= ((w-4,1),(2,1),\vn),\\
\la^{(3)} &= ((w-4,1),(1^3),\vn), &
\la^{(4)} &= ((w-4),(2,1^2),\vn).
\end{alignat*}
If $p=3$, then we set
\begin{alignat*}{3}
\la^{(1)} &= ((w-4,2^2),\vn^2), \qquad &\la^{(2)} &= ((w-4,2,1^2),\vn^2),\qquad 
&\la^{(3)} &= ((w-4,2,1),(1),\vn),\\
\la^{(4)} &= ((w-4,2),(2),\vn), &
\la^{(5)} &= ((w-4,2),(1^2),\vn), &\la^{(6)}&=((w-4,1),(1^3),\vn).
\end{alignat*}
In either case, we may remove the first row of each partition and reduce the computation of decomposition numbers to the case $B([1,s_1,s_1+2], 4)$, which is Scopes equivalent to the block $B([1,4,6],4)$, handled in \cref{prop:NearRouqblocks}.
The result follows immediately, since the remaining partitions are precisely those used in that proposition.

Next, we suppose that $s_2-s_1\geq 4$, so that we also have $s_1 \geq 5$.
We take the same partitions as in the previous case, and removing the first row now yields partitions in the block $B([1,s_1,s_2-1], 4)$, which is Scopes equivalent to $B([1,4,7], 4)$, and the result follows as before.

It remains to check the blocks for which $s_2<s_1$, which we now check.
Since $[1,s_1,s_2]' = [1,s_1,s_1-s_2]$, we may assume that $s_2 \geq s_1-s_2$.
Since $p_{e-1}-p_{e-2}>2e$, we may also assume that $s_1-s_2 \geq 3$.

If $s_1-s_2 = 3$, then $s_2 \geq 3$.
We define partitions $\la^{(0)},\la^{(1)}, \la^{(2)}, \la^{(3)}, \la^{(4)}$ as in the $s_1\leq s_2$ case we considered above with $s_2-s_1=2$.
Removing the first row yields partitions in the block $B([1,s_1-1,s_2], 3)$, which is Scopes equivalent to $B([1,3,4], 3)$, which is why the chosen partitions work again.

Finally, we suppose that $s_1-s_2\geq4$ so that $s_2 \geq 4$.
Then we define partitions as in the $s_1\leq s_2$ cases where $s_2-s_1 \geq 3$ above, and removing the first row yields partitions in the block $B([1,s_1-1,s_2], 4)$, which is Scopes equivalent to $B([1, 4, 6], 4)$ if $s_1 - s_2 = 4$, or $B([1,4,7], 4)$ if $s_1 - s_2 \geq 5$, and the result follows.
\end{proof}

Finally, we use \cref{thm:MAIN} to prove \cref{cor:schuralg}.

\begin{proof}[Proof of \cref{cor:schuralg}]
If $B$ is not Scopes equivalent to the weight $2$ Rouquier block with $e=3$, then the result follows since the corresponding block of $\hhh$ is strictly wild, and is an idempotent truncation of $B$.
To see this, take the usual Schur functor $F: B{-}\mathrm{mod} \rightarrow eBe{-}\mathrm{mod}$ given by $M\mapsto eM$, and the functor $G:eBe{-}\mathrm{mod} \rightarrow B{-}\mathrm{mod}$ given by $N \mapsto Be \otimes_{eBe} N$.
Then $FG(N) \cong N$, and one may use this to show that $G$ is a full representation embedding.

Now suppose that $B$ is the weight $2$ Rouquier block with $e=3$, i.e.~the block of $S_q(11,11)$ with core $(3,1^2)$.
Then we may check that the submatrix indexed by partitions $(6,4,1)$, $(6,3,2)$, $(6,1^5)$, $(3^2,2^2,1)$, and $(3^2,2,1^3)$ is of the form~(\ref{targetmatrix}) and is characteristic-free.
Indeed, for $p\neq2$ this follows from \cite[Theorem~3.5 and Corollary~3.6]{schrolltan07}.
By \cref{thm:wt2adjust}, we may compute the graded decomposition numbers for the corresponding block of $\hhh[11]$, and therefore all $d_{\la\mu}^{e,2}(v)$ for $\la,\mu \in B$ with $\mu$ $3$-regular.
Since $(6,1^5)$ and $(3^2,2,1^3)$ are $3$-singular, a separate check is required.
We may use the Jantzen sum formula to confirm that the necessary decomposition numbers remain the same in characteristics $0$ and $2$.
The proof of \cref{prop:matrixtrick} easily applies to $S_q(n,n)$, and the result follows.
\begin{answer}
We give the submatrix explicitly below.
\[
\begin{array}{r|ccccc}
&\rt{6,4,1}
&\rt{6,3,2}
&\rt{6,1^5}
&\rt{3^2,2^2,1}
&\rt{3^2,2,1^3}
\\\hline
6,4,1&1&\cdot&\cdot&\cdot&\cdot\\
6,3,2&v&1&\cdot&\cdot&\cdot\\
6,1^5&\cdot&v&1&\cdot&\cdot\\
3^2,2^2,1&\cdot&v&\cdot&1&\cdot\\
3^2,2,1^3&\cdot&v^2&v&v&1
\end{array}\qedhere
\]
\end{answer}
\end{proof}

\bibliographystyle{lspaper}
\addcontentsline{toc}{section}{\refname}
\bibliography{master}

\newcommand{\etalchar}[1]{$^{#1}$}
\providecommand{\bysame}{\leavevmode\hbox to3em{\hrulefill}\thinspace}
\providecommand{\MR}{\relax\ifhmode\unskip\space\fi MR }
\providecommand{\MRhref}[2]{%
  \href{http://www.ams.org/mathscinet-getitem?mr=#1}{#2}
}
\providecommand{\href}[2]{#2}
\begin{thebibliography}{CCC{\etalchar{+}}10}

\bibitem[ALS23]{als23}
S.~Ariki, S.~Lyle, and L.~Speyer,
  \emph{\href{http://dx.doi.org/10.1112/jlms.12808}{Schurian-finiteness of
  blocks of type {$A$} {Hecke} algebras}}, J.\ Lond.\ Math.\ Soc.\ (2)
  \textbf{108} (2023), no.~6, 2333--2376.

\bibitem[BD10]{BDbricktame}
L.~Bodnarchuk and Y.~Drozd,
  \emph{\href{http://dx.doi.org/10.1016/j.jalgebra.2010.02.030}{One class of
  wild but brick-tame matrix problems}}, J.\ Algebra \textbf{323} (2010),
  no.~10, 3004--3019.

\bibitem[BK09]{bk09}
J.~Brundan and A.~S. Kleshchev,
  \emph{\href{http://dx.doi.org/10.1016/j.aim.2009.06.018}{Graded decomposition
  numbers for cyclotomic {Hecke} algebras}}, Adv.\ Math. \textbf{222} (2009),
  no.~6, 1883--1942.

\bibitem[CC15]{CCSchurtame}
A.~T. Carroll and C.~Chindris,
  \emph{\href{http://dx.doi.org/10.1007/s10468-015-9527-x}{Moduli spaces of
  modules of {Schur}-tame algebras}}, Algebr.\ Represent.\ Theory \textbf{18}
  (2015), no.~4, 961--976.

\bibitem[CCC{\etalchar{+}}10]{hyperbolicDynkin}
L.~Carbone, S.~Chung, L.~Cobbs, R.~McRae, D.~Nandi, Y.~Naqvi, and D.~Penta,
  \emph{\href{http://dx.doi.org/10.1088/1751-8113/43/15/155209}{Classification
  of hyperbolic {Dynkin} diagrams, root lengths and {Weyl} group orbits}}, J.\
  Phys.\ A \textbf{43} (2010), no.~15, 155209, 30.

\bibitem[CMT02]{cmt02}
J.~Chuang, H.~Miyachi, and K.~M. Tan,
  \emph{\href{http://dx.doi.org/10.1016/S0021-8693(02)00062-5}{Row and column
  removal in the {$q$}-deformed {Fock} space}}, J.\ Algebra \textbf{254}
  (2002), no.~1, 84--91.

\bibitem[CT01]{ct01}
J.~Chuang and K.~M. Tan,
  \emph{\href{http://dx.doi.org/10.1112/blms/33.2.157}{On certain blocks of
  {Schur} algebras}}, Bull.\ London Math.\ Soc. \textbf{33} (2001), no.~2,
  157--167.

\bibitem[DJ89]{dj89}
R.~Dipper and G.~D. James,
  \emph{\href{http://dx.doi.org/10.1112/plms/s3-59.1.23}{The {$q$}-{Schur}
  algebra}}, Proc.\ London Math.\ Soc.\ (3) \textbf{59} (1989), no.~1, 23--50.

\bibitem[EN02]{enreptype}
K.~Erdmann and D.~K. Nakano,
  \emph{\href{http://dx.doi.org/10.1090/S0002-9947-01-02848-3}{Representation
  type of {Hecke} algebras of type {$A$}}}, Trans.\ Amer.\ Math.\ Soc.
  \textbf{354} (2002), no.~1, 275--285.

\bibitem[Fay05]{faywt2}
M.~Fayers, \emph{\href{http://dx.doi.org/10.1017/S0305004105008637}{Weight two
  blocks of {Iwahori}--{Hecke} algebras in characteristic two}}, Math.\ Proc.\
  Cambridge Philos.\ Soc. \textbf{139} (2005), no.~3, 385--397.

\bibitem[Fay07]{fay07wt4}
\bysame, \emph{\href{http://dx.doi.org/10.1016/j.jalgebra.2007.08.006}{James's
  {Conjecture} holds for weight four blocks of {Iwahori}--{Hecke} algebras}},
  J.\ Algebra \textbf{317} (2007), no.~2, 593--633.

\bibitem[Fay08]{fay08wt3}
\bysame,
  \emph{\href{http://dx.doi.org/10.1090/S0002-9947-07-04156-6}{Decomposition
  numbers for weight three blocks of symmetric groups and {Iwahori}--{Hecke}
  algebras}}, Trans.\ Amer.\ Math.\ Soc. \textbf{360} (2008), no.~3,
  1341--1376.

\bibitem[FT06]{faytan06}
M.~Fayers and K.~M. Tan,
  \emph{\href{http://dx.doi.org/10.1016/j.jalgebra.2006.01.054}{Adjustment
  matrices for weight three blocks of {Iwahori}--{Hecke} algebras}}, J.\
  Algebra \textbf{306} (2006), no.~1, 76--103.

\bibitem[Jam81]{j81}
G.~D. James, \emph{\href{http://dx.doi.org/10.1016/0021-8693(81)90108-3}{On the
  decomposition matrices of the symmetric groups, {III}}}, J.\ Algebra
  \textbf{71} (1981), no.~1, 115--122.

\bibitem[JLM06]{jlm}
G.~D. James, S.~Lyle, and A.~Mathas,
  \emph{\href{http://dx.doi.org/10.1007/s00209-005-0863-0}{Rouquier blocks}},
  Math.\ Z. \textbf{252} (2006), no.~3, 511--531.

\bibitem[JM02]{jm02}
G.~D. James and A.~Mathas,
  \emph{\href{http://dx.doi.org/10.1016/S0021-8693(02)00644-0}{Equating
  decomposition numbers for different primes}}, J.\ Algebra \textbf{258}
  (2002), no.~2, 599--614.

\bibitem[LLT96]{LLT}
A.~Lascoux, B.~Leclerc, and J.-Y. Thibon,
  \emph{\href{http://projecteuclid.org/euclid.cmp/1104287629}{Hecke algebras at
  roots of unity and crystal bases of quantum affine algebras}}, Comm.\ Math.\
  Phys. \textbf{181} (1996), no.~1, 205--263.

\bibitem[Mat99]{mathas}
A.~Mathas, \emph{\href{http://dx.doi.org/10.1090/ulect/015}{Iwahori--{Hecke}
  algebras and {Schur} algebras of the symmetric group}}, University Lecture
  Series, vol.~15, American Mathematical Society, Providence, RI, 1999.

\bibitem[Ric96]{richardswt2}
M.~J. Richards, \emph{\href{http://dx.doi.org/10.1017/S0305004100074296}{Some
  decomposition numbers for {Hecke} algebras of general linear groups}}, Math.\
  Proc.\ Cambridge Philos.\ Soc. \textbf{119} (1996), no.~3, 383--402.

\bibitem[SS07]{simsonskow3}
D.~Simson and A.~Skowro\'{n}ski,
  \emph{\href{http://dx.doi.org/10.1017/CBO9780511619403}{Elements of the
  representation theory of associative algebras. {V}ol.~3}}, London
  Mathematical Society Student Texts, vol.~72, Cambridge University Press,
  Cambridge, 2007, Representation-infinite tilted algebras.

\bibitem[ST07]{schrolltan07}
S.~Schroll and K.~M. Tan,
  \emph{\href{http://dx.doi.org/10.1093/imrn/rnm130}{Weight 2 blocks of general
  linear groups and modular {Alvis}--{Curtis} duality}}, Int.\ Math.\ Res.\
  Not.\ IMRN (2007), no.~24, Art.\ ID rnm130, 35.

\end{thebibliography}

\end{document}